\documentclass[12pt,oneside]{amsart}
\usepackage[doi=false,isbn=false,url=false,style=alphabetic]{biblatex}
\bibliography{citations.bib}
\input{commands.tex}
\usepackage{etoolbox}
\usepackage{enumitem}[cleveref]
\usepackage{amsfonts}
\usepackage{float}
\usepackage{epigraph}
\usepackage{xcolor}
\definecolor{darkgreen}{rgb}{0,0.30,0}

\definecolor{darkred}{rgb}{0.75,0,0}

\definecolor{darkblue}{rgb}{0,0,0.6} 

\usepackage[colorlinks,citecolor=darkred,linkcolor=darkred,urlcolor=darkblue]{hyperref}
\usepackage{cleveref}

\def\makeautorefname#1#2{\expandafter\def\csname#1autorefname\endcsname{#2}}

\makeautorefname{eqn}{Equation}%
\makeautorefname{sec}{Section}%
\makeautorefname{subsec}{Subsection}%
\makeautorefname{footnote}{footnote}%
\makeautorefname{item}{item}%
\makeautorefname{figure}{Figure}%
\makeautorefname{table}{Table}%
\makeautorefname{wraptab}{wraptable}%
\makeautorefname{part}{Part}%
\makeautorefname{app}{Appendix}%
\makeautorefname{cla}{claim}%
\makeautorefname{assump}{assumption}%
\makeautorefname{conj}{conjecture}%
\makeautorefname{cor}{corollary}%
\makeautorefname{conv}{convention}%
\makeautorefname{cex}{counterexample}%
\makeautorefname{cexs}{counterexamples}%
\makeautorefname{dig}{digression}%
\makeautorefname{disc}{discussion}%
\makeautorefname{def}{definition}%
\makeautorefname{ex}{example}%
\makeautorefname{exs}{examples}%
\makeautorefname{fac}{fact}%
\makeautorefname{intu}{intuition}%
\makeautorefname{lem}{lemma}%
\makeautorefname{meta}{metathm}%
\makeautorefname{nota}{notation}%
\makeautorefname{note}{note}%
\makeautorefname{warn}{warn}%
\makeautorefname{prop}{proposition}%
\makeautorefname{rmk}{remark}%
\makeautorefname{strat}{strategy}%
\makeautorefname{term}{terminology}%
\makeautorefname{thm}{theorem}%
\makeautorefname{upsh}{upshot}%
\theoremstyle{definition}
\numberwithin{equation}{section}
\newtheorem{theorem}[equation]{Theorem}

\newtheorem{convention}[equation]{Convention}
\newtheorem{corollary}[equation]{Corollary}

\newtheorem{definition}[equation]{Definition}

\newtheorem{example}[equation]{Example}

\newtheorem{lemma}[equation]{Lemma}

\newtheorem{notation}[equation]{Notation}

\newtheorem{proposition}[equation]{Proposition}
\newtheorem{remark}[equation]{Remark}

\makeatletter
\let\c@corollary=\c@theorem
\let\c@proposition=\c@theorem
\let\c@lemma=\c@theorem
\let\c@assumption=\c@theorem
\let\c@conjecture=\c@theorem
\let\c@definition=\c@theorem
\let\c@example=\c@theorem
\let\c@remark=\c@theorem
\let\c@notation=\c@theorem
\let\c@equation\c@theorem
\let\c@strategy\c@theorem
\makeatother

\let\sqtimes\boxtimes

\providecommand{\Pic}{\text{Pic}}

\renewcommand{\O}{\mathcal{O}}
\providecommand{\Th}{\text{Th}}

\providecommand{\MW}{\text{MW}}

\let\del\partial

\providecommand{\Sm}{\text{Sm}}

\providecommand{\A}{\mathbb{A}}

\renewcommand{\Ch}{\text{Ch}}
\providecommand{\CHW}{\widetilde{\text{CH}}}
\providecommand{\BSp}{\text{BSp}}

\providecommand{\BGm}{\text{B}\mathbb{G}_{\rm m}}
\providecommand{\et}{\text{\'{e}t}}

\providecommand{\Bdet}{\text{Bdet}}

\renewcommand{\O}{\mathcal{O}}

\title[Chow--Witt rings for metalinear bundles]{The {C}how--{W}itt rings of the classifying spaces of quadratically oriented bundles}

\author{Thomas Brazelton and Matthias Wendt}
\date{\today}

\begin{document}

\begin{abstract}
In this paper we compute the Chow--Witt rings of the classifying space ${\rm BSL}_n^c$ of quadratically oriented vector bundles of rank $n$. We also discuss the corresponding quadratically-oriented cobordism spectrum ${\rm MSL}^c$ and show that it is equivalent to $\rm{MSL}$ after inverting $\eta$.
\end{abstract}
\maketitle

\setcounter{tocdepth}{1}
\tableofcontents{}

\section{Introduction}

\textit{Quadratically oriented vector bundles}\footnote{Also called \textit{metalinear vector bundles} in \cite[\S 3.3.2]{AHW2}.}, or vector bundles whose determinant admits a square root, are natural objects of study due to their importance in quadratically enriched enumerative geometry. The theory of motivic cohomology theories which support quadratic orientations, so-called $\SL^c$\textit{-oriented} theories, is well-studied in motivic homotopy theory \cite{Ananyevskiy}. Quadratically oriented bundles are further related to \textit{relatively oriented} vector bundles, which are well-studied objects in motivic homotopy theory \cite{BM,Fasel-CHW,Morel,AF-comparing,EulerBW3} as well as in real enumerative geometry \cite{OkTel}.

With these perspectives in mind, the classifying space for quadratically oriented vector bundles $\BSL_n^c$ is a natural object to investigate, one reason being that its cohomology is the natural home for characteristic classes of quadratically oriented bundles.

As usual when dealing with torsors and classifying spaces in an algebrao-geometric or motivic setting, there are different topologies and thus potentially different classifying spaces to consider. Fortunately, the group $\SL_n^c$ is special in the sense of Serre, i.e., \'etale-locally trivial torsors are already Zariski-locally trivial, which means that the usual classifying spaces are all equivalent, cf.~\autoref{prop:slnc-special} and \autoref{cor:real-realization-bslnc}:
\[
  {\rm B}_{\rm Zar}\SL_n^c\simeq{\rm B}_{\rm Nis}\SL_n^c\simeq{\rm B}_\et\SL_n^c
\]
For this reason, we drop notational distinctions between these different versions and denote \emph{the} classifying space of $\SL_n^c$ simply by $\BSL_n^c$. 

As the classifying space of quadratically oriented vector bundles, the space $\BSL_n^c$ is an interpolation between $\BSL_n$ and $\BGL_n$; the properties of being oriented and quadratically oriented are indistinguishable over $\R$ (although the additional information of the square-root line bundle is still visible over $\R$). This indicates a similarity between $\BSL_n^c$ and $\BSL_n$ which we make precise in the language of ${\bf I}^j$-cohomology. On the other hand, $\BSL_n^c$ is not $\A^1$-simply connected, instead its fundamental group agrees with that of $\BGL_n$, and we might expect some similarity between $\BSL_n^c$ and $\BGL_n$, which will be visible in untwisted Witt-sheaf cohomology.

In this paper we compute the Chow--Witt rings of $\BSL_n^c$. Our methods of computation are heavily inspired by \cite{HornbostelWendt} and \cite{Wendt-Gr}. The main result is the following, see \autoref{sec:pf-main-thm}.

\begin{theorem}
  \label{thm:CHW-computation}
  Let $k$ be a field of characteristic $\neq 2$. 
  \begin{enumerate}
  \item The Chow--Witt groups of $\BSL_n^c$ are described by a pullback square
    \[
      \begin{tikzcd}
        \CHW^j(\BSL_n^c,\mathcal{L})\rar\dar\pb & \ker(\del_{j,\mathcal{L}})\dar\\
        H^j(\BSL_n^c,\mathbf{I}^j(\mathcal{L}))\rar["\rho" below] & \Ch^j(\BSL_n^c),
      \end{tikzcd}
    \]
    Here, the possible twists are $\mathcal{L}=\mathcal{O},\mathcal{O}_{\mathbb{P}^\infty}(-1)$, the trivial bundle and the square-root of the determinant line bundle\footnote{For most of the paper, we will also denote the square-root line bundle by $\Theta$, to evoke the link to theta-characteristics.} providing the orientation of the universal vector bundle. The lower left corner is cohomology of the fundamental ideal sheaf, in the upper right the boundary map is
    \[
    \partial_{\mathcal{L},j}\colon \CH^j(\BSL_n^c)\xrightarrow{\bmod 2}\Ch^j(\BSL_n^c)\xrightarrow{\beta_{\mathcal{L}}} H^{j+1}(\BSL_n^c,{\bf I}^{j+1}(\mathcal{L})).
    \]
  \item The Chow ring of $\BSL_n^c$ is described in \autoref{cor:chow-product}, the ${\bf I}$-cohomology ring is described in \autoref{thm:I-cohom-bslnc}, and the relevant kernels of boundary maps $\partial_{j,\mathcal{L}}$ are described in \autoref{lem:ker-bockstein-untwisted} and \autoref{lem:ker-bockstein-twisted}. 
  \item The Chow--Witt ring is generated by the following classes
    \begin{itemize}
    \item the even \emph{Pontryagin classes} $p_{2i}\in\CHW^{4i}(\BSL_n^c,\mathcal{O})$ for $i=1,\dots,\lfloor\frac{n-1}{2}\rfloor$, which in terms of their Witt-sheaf and Chow contributions can be written as
      \[
      p_{2i}=\left( p_{2i}, c_{2i}^2 + 2\sum_{j=\max{0,4i-n}}^{2i-1} (-1)^j c_j c_{4i-j} \right),
      \]
    \item the \emph{Euler class} $e_n\in\CHW^n(\BSL_n^c,\mathcal{O})$ for even $n$, which in the fiber product picture can be described as $e_n = (e_n,c_n)$,
    \item the \emph{Bockstein classes}
      \[
      \widetilde{\beta}_{\mathcal{L}}(\overline{c}_J)=\widetilde{\beta}_{\mathcal{L}}(\overline{c}_{2j_1}\cdots\overline{c}_{2j_l})\in\CHW^{1+\sum_{i=1}^l2j_i}(\BSL_n^c,\mathcal{L})
      \]
      for index sets $J=\left\{0<j_1<\cdots <j_l\leq\lfloor\frac{n-1}{2}\rfloor\right\}$ which can be empty if $\mathcal{L}$ is the nontrivial twist,
    \item the \emph{hyperbolic Chern classes} $H_{\mathcal{L}}(x)\in\CHW^q(\BSL_n^c,\mathcal{L})$ for $x\in \CH^q(\BSL_n^c)$. 

\end{itemize}
  \item The pullback is compatible with multiplicative structures, i.e., the fiber product description in (1) can be used to reduce computations of products of classes in Chow--Witt theory to computations in ${\bf I}$-cohomology and Chow theory. 
  \end{enumerate}
 \end{theorem}

The above theorem provides all the information to do computations in Chow--Witt rings of $\BSL_n^c$. We don't give a full list of relations, because in particular the hyperbolic Chern classes make this a painful task while not producing many additional insights.
To prove the result, we first provide in 
% Section 2
\autoref{sec:preliminaries} an introduction to many of our key players, including Chow--Witt groups, Bockstein homomorphisms and motivic Steenrod squares, as well as K\"{u}nneth formulas for Chow and Witt-sheaf cohomology.
% Section 3
In \autoref{sec:BSLnc} we construct and discuss the classifying space $\BSL_n^c$ of metalinear vector bundles.
% Section 4
In \autoref{sec:oriented-theories} we compute various oriented cohomologies of $\BSL_n^c$, namely its Chow groups, Witt sheaf cohomology, and motivic cohomology, and discuss the action of Steenrod squares.
% Section 5
In \autoref{sec:chow-witt} we compute the $\mathbf{I}^j$-cohomology of $\BSL_n^c$, which allows us to prove our main theorem on the Chow--Witt groups of $\BSL_n^c$.

% Section 6
In \autoref{sec:real-realizations} we explore the real Betti realizations of $\BSL_n^c$ and the associated Thom spectrum $\MSL^c$, from which we can show the following result (cf.~\autoref{cor:real-realization-bslnc} and \autoref{prop:realization-mslnc}).

\begin{proposition}
  The natural morphism ${\rm SL}_n\to{\rm SL}_n^c$ induces equivalences
  \begin{align*}
    {\rm Re}_{\mathbb{R}}\left(\BSL_n\right)[1/2]&\simeq{\rm Re}_{\mathbb{R}}\left(\BSL_n^c\right)[1/2]\\
    {\rm Re}_{\mathbb{R}}\left({\rm MSL}\right)[1/2]&\simeq{\rm Re}_{\mathbb{R}}\left({\rm MSL}^c\right)[1/2].
  \end{align*}
\end{proposition}
Note that ${\rm Re}_{\mathbb{R}}\left(\BSL_n\right)\simeq\BSO_n$ and ${\rm Re}_{\mathbb{R}}\left(\MSL\right)\simeq{\rm MSO}$.

% Section 7
Finally, as a consequence of this and the cohomology computations in this paper, we show in \autoref{sec:MSL-MSLc} that the Thom spectra $\MSL$ and $\MSL^c$ become equivalent after inverting $\eta$. This fact is implicit in much of the literature comparing $\SL$ and ${\rm SL}^c$-orientations, but to the best of our knowledge hasn't been noted explicitly. The statement is proved as \autoref{cor:comparison-mslc}, and the proof techniques owe much to the work of Bachmann and Hopkins \cite{bachmann:hopkins}.

\begin{corollary}
  Let $k$ be a field of characteristic $\ne 2$. The natural morphism ${\rm MSL}\to {\rm MSL}^c$ induces an equivalence in the $\eta$-inverted motivic stable homotopy category $\mathcal{SH}(k)[\eta^{-1}]$:
  \[
    {\rm MSL}[\eta^{-1}]\xrightarrow{\simeq}{\rm MSL}^c[\eta^{-1}]
  \]
\end{corollary}

\subsection{Acknowledgements}

The first named author is supported by an NSF Postdoctoral Fellowship (DMS-2303242).

\section{Preliminaries}\label{sec:preliminaries}

We provide a rough introduction to Chow--Witt groups, the cohomology of $\A^1$-invariant sheaves, and various long exact sequences and computational techniques.

\begin{convention} All cohomology groups $H^n(X,\mathscr{F})$ are Nisnevich cohomology of a strictly $\A^1$-invariant sheaf $\mathscr{F}$ of abelian groups, unless otherwise specified. We will use $\bullet$ to denote $\Z$-indexed cohomology theories. Finally we will denote by $\mathscr{F}(\mathcal{L})$ the twist of a sheaf by a line bundle.
\end{convention}

\subsection{Chow--Witt groups}

By Morel \cite{Morel}, there is a well-known cartesian square of strictly $\A^1$-invariant Nisnevich sheaves
\begin{equation}\label{eqn:pullback-sheaves}
\begin{aligned}
    \begin{tikzcd}[ampersand replacement=\&]
    \mathbf{K}_n^\MW\rar\dar\pb \& \mathbf{K}_n^{\rm M}\dar\\
    \mathbf{I}^n\rar \& \mathbf{K}^{\rm M}_n/2.
\end{tikzcd}
\end{aligned}
\end{equation}
The fibers of the horizontal arrows agree by virtue of this square being cartesian, and the same is true for the vertical arrows, inducing four distinct short exact sequences of sheaves
\begin{equation}\label{eqn:ses-of-sheaves}
\begin{aligned}
    \mathbf{I}^{n+1} \to &\mathbf{K}_n^\MW \to \mathbf{K}_n^{\rm M}, \\
    \mathbf{I}^{n+1} \to &\mathbf{I}^n \to \mathbf{K}_n^{\rm M}/2, \\
    2\mathbf{K}_n^{\rm M} \to &\mathbf{K}_n^\MW \to \mathbf{I}^n, \\
    2\mathbf{K}_n^{\rm M} \to &\mathbf{K}_n^{\rm M} \to \mathbf{K}_n^{\rm M}/2.
\end{aligned}
\end{equation}
These provide long exact sequences on Nisnevich cohomology.

\begin{example}\label{ex:chw} The \textit{Chow--Witt groups} $\CHW^n(X)$ are defined as $H^n(X, \mathbf{K}_n^\MW)$. Since the Chow groups are similarly computable as the cohomology of Milnor $K$-theory (Bloch's formula), we obtain long exact sequences of the form
\begin{align*}
    \cdots\to H^n(X,\mathbf{I}^{n+1}) \to \CHW^n(X) \to \CH^n(X) \xto{\del} H^{n+1}(X, \mathbf{I}^{n+1})\to \cdots
\end{align*}
\end{example}

\begin{remark} The oriented intersection product
\begin{align*}
    \mathbf{I}^a \times \mathbf{I}^b \to \mathbf{I}^{a+b}
\end{align*}
turns $H^\bullet(X, \mathbf{I}^\bullet)$ into a graded ring, and $\del$ can be thought of as a graded abelian group homomorphism
\begin{align*}
    \del \colon \CH^\bullet(X) \to H^{\bullet+1}(X, \mathbf{I}^{\bullet+1}).
\end{align*}
Moreover, $\del$ is a derivation, hence $\ker(\del)$ is a subring of $\CH^\bullet(X)$.
\end{remark}

\begin{notation} We denote the Chow groups modulo two by $\Ch^n(X)$.
\end{notation}

The pullback square \autoref{eqn:pullback-sheaves} induces a map from Chow--Witt theory to the pullback of ${\bf I}^j$-cohomology with the kernel of $\del$ (a subgroup of the Chow groups), and we may ask how far it is from being an isomorphism. The following result gives us some settings where this is indeed an isomorphism.

\begin{proposition} \cite[Prop.~2.11]{HornbostelWendt} For $X$ a smooth $F$-scheme, the canonical ring hom
\begin{align*}
    \CHW^\bullet(X) \to H^\bullet(X, \mathbf{I}^\bullet) \times_{\Ch^\bullet(X)} \ker(\del)
\end{align*}
is a surjective ring homomorphism, which is injective if either
\begin{enumerate}
    \item $\CH^\bullet(X)$ has no non-trivial 2-torsion
    \item the map $\eta\colon H^n(X, \mathbf{I}^{n+1}) \to H^n(X,\mathbf{I}^n)$ is injective.
\end{enumerate}
\end{proposition}

\subsection{Bockstein homomorphisms} Recall that the short exact sequence of abelian groups
\begin{equation}\label{eqn:mult-2-SES}
\begin{aligned}
   0 \to \Z \xto{2} \Z \to \Z/2\Z \to 0 
\end{aligned}
\end{equation}
induces a long exact sequence on singular cohomology
\begin{align*}
    \cdots \to H^q(X;\Z) \xto{\rho} H^q(X;\Z/2\Z) \xto{\beta} H^{q+1}(X;\Z) \xto{2} H^{q+1}(X;\Z) \to \cdots
\end{align*}
whose boundary map is the \emph{integral Bockstein} morphism $\beta$. There is a similar boundary map in the long exact sequence associated to the short exact sequence
\begin{align*}
  0\to\Z/2\Z\to\Z/4\Z\to\Z/2\Z\to 0;
\end{align*}
this boundary map is the first Steenrod square $\Sq^1\colon H^q(X;\Z/2\Z)\to H^{q+1}(X;\Z/2\Z)$.  By definition, the integral Bockstein map and the first Steenrod square $\Sq^1$ are related by the commutative diagram
\[ \begin{tikzcd}
    H^q(X;\Z/2\Z)\rar["\beta"]\ar[dr,"\Sq^1" below left] & H^{q+1}(X;\Z)\dar["\rho" right]\\
     & H^{q+1}(X;\Z/2\Z).
\end{tikzcd} \]
Since $\Sq^1$ isn't a ring homomorphism, but only a derivation, it's hard to describe its image and kernel directly from its definition. The following result, initially due to Brown, allows us to circumvent this difficulty under some assumptions about the torsion. This works for any prime, but we state it for the prime $2$ here.

\begin{lemma}\label{lem:2-torsion-lemma} \cite[Lemma~2.2]{Brown}, \cite[Corollary~3E.4]{Hatcher} Suppose $H^\bullet(X;\Z)$ has no 4-torsion. Then $\rho\colon H^\bullet(X;\Z) \to H^\bullet(X;\Z/2\Z)$ is injective on the $2$-torsion, the image of the $2$-torsion under $\rho$ is $\im(\Sq^1)$, and we have that $\ker(\beta) = \ker(\Sq^1)$.
\end{lemma}

This is a key result in the inductive computation of the integral cohomology of $\BO_n$ and $\BSO_n$ in \cite{Brown}. We obtain an analogous story in the motivic context. The following short exact sequence of sheaves is intended to be reminiscent of \autoref{eqn:mult-2-SES}:
\begin{align*}
    0 \to \mathbf{I}^{n+1} \to \mathbf{I}^n \to \mathbf{K}_n^{\rm M}/2\to 0,
\end{align*}
and it admits a twisted generalization for any line bundle $\mathcal{L}$ over a base space $X$:
\begin{align*}
    0 \to \mathbf{I}^{n+1}(\mathcal{L}) \to \mathbf{I}^{n}(\mathcal{L}) \to \mathbf{K}_n^{\rm M}/2\to 0.
\end{align*}
This induces a long exact sequence of cohomology groups, often called the \textit{B\"ar sequence}:
\begin{align*}
    \cdots \to H^q(X,\mathbf{I}^{n+1}(\mathcal{L})) \xto{\eta} H^q(X,\mathbf{I}^{n}(\mathcal{L})) \xto{\rho_\mathcal{L}} H^{q}(X, \mathbf{K}_n^{\rm M}/2)\xto{\beta_\mathcal{L}} H^{q+1}(X,\mathbf{I}^{n+1}(\mathcal{L}))\to \cdots 
\end{align*}
Here the map $\eta$ is multiplication by the Hopf element $\eta\in \mathbf{K}_{-1}^\MW$, corresponding to multiplication by 2 in the topological setting above. When $q=n$, we get that $H^n(X,\mathbf{K}_n^{\rm M}/2) \cong \Ch^n(X)$ is the mod 2 Chow group of $X$.

\begin{remark}\label{rmk:labelname} When the cohomological index on the B\"ar sequence starts to reach the exponent on the fundamental ideal, the nature of the sequence changes. In particular, since negative powers of the fundamental ideal are by convention the Witt sheaf, the Gersten resolution for $\mathbf{I}^j$-cohomology identifies the cohomology with Witt-sheaf cohomology in the following sense:
\begin{align*}
    H^n (X, \mathbf{I}^j(\mathcal{L})) \xto{\sim} H^n(X, \mathbf{W}(\mathcal{L})) \text{ for } n>j.
\end{align*}
Moreover, negative  mod 2 Milnor $K$-theory vanishes,
% \footnote{\todo this is true for ordinary (non mod 2) Milnor $K$-theory as well right?} yes, negative Milnor K-theory is trivial
so again a Gersten resolution argument easily shows that
\begin{align*}
    H^n(X, \mathbf{K}_j^{\rm M}/2) \cong 0 \text{ for } n>j.
\end{align*}
These observations together give an exact sequence exhibiting Witt-sheaf cohomology as a quotient of ${\bf I}$-cohomology:
\begin{equation}\label{eqn:four-term-Bar}
\begin{aligned}
    \Ch^{j}(X) \xto{\beta_\mathcal{L}} H^{j+1}(X, \mathbf{I}^{j+1}(\mathcal{L})) \xto{\eta} H^{j+1}(X, \mathbf{W}(\mathcal{L})) \to 0.
\end{aligned}
\end{equation}
\end{remark}

\begin{notation}\label{nota:beta-subscript} We remark a few notational conventions.
\begin{enumerate}
    \item When the line bundle $\mathcal{L}$ is trivial, we drop the subscript and simply write $\beta$ and $\rho$ instead of $\beta_\mathcal{L}$ and $\rho_\mathcal{L}$.
    \item When the cohomological degree matches the grading on the ${\bf I}^j$ and Milnor $K$-theory sheaves, we will often add a subscript on the Bockstein and reduction homomorphisms:
\begin{align*}
    \beta_j \colon  H^j(X,\mathbf{K}_j^{\rm M}/2) &\to H^{j+1}(X; \mathbf{I}^{j+1}) \\
    \rho_j \colon  H^j(X,\mathbf{I}^j) &\to H^j(X,\mathbf{K}_j^{\rm M}/2).
\end{align*}
This will be important later, as we will need to keep track of two distinct Bocksteins in different degrees for a computation. As the subscript on $\beta$ is often reserved for twists of line bundles, when we are both twisting and being cautious about degrees we will unfortunately need to denote this by $\beta_{\mathcal{L},j}$.
\end{enumerate}
\end{notation}

Just as in the classical context, we can piece together the maps in the B\"ar sequence to obtain a motivic Steenrod square. This identification is due to Totaro \cite[Theorem~1.1]{Totaro-Witt}, and the twisted version is due to Asok and Fasel \cite[Theorem~3.4.1]{asok:fasel:secondary}. 

\begin{proposition}\label{prop:totaro-sq2} The composite
\begin{align*}
    \Ch^n(X) \xto{\beta_{\mathcal{L},n}} H^{n+1}(X,\mathbf{I}^{n+1}(\mathcal{L})) \xto{\rho_{\mathcal{L},n+1}} \Ch^{n+1}(X)
\end{align*}
is the motivic Steenrod square $\Sq^2_{\mathcal{L}}$ (note that we may also denote this $\Sq_{\mathcal{L},n}^2$ to be consistent with \autoref{nota:beta-subscript}).
\end{proposition}

The Steenrod square satisfies $\Sq^2_\mathcal{L} \circ \Sq^2_\mathcal{L} = 0$ by the B\"ar sequence, and it is a derivation satisfying the Jacobi identity by e.g. \cite[Lemma~2.10]{HornbostelWendt}. If we'd like to characterize the image or kernel of $\Sq^2$ on the mod two Chow groups of a $k$-variety, we might want an analogue of \autoref{lem:2-torsion-lemma} in the motivic setting. Note that in Brown's original paper, he operated under the assumption that all torsion in the integral cohomology be 2-torsion. This condition can be weakened to just require that there is no 4-torsion. We can similarly strengthen the assumption that all torsion is $\eta$-torsion in $H^\bullet(X,\mathbf{I}^\ast)$ to the assumption that there is no $\eta^2$ torsion, obtaining a strengthening of \cite[Lemma~2.4]{Wendt-Gr}.

\begin{lemma}\label{lem:eta-squared-torsion} Let $X$ be a smooth scheme over a field of characteristic $\ne 2$%, and let $\mathcal{L}$ be a line bundle on $X$
. If $H^n(X,\mathbf{I}^{n})$ has no $\eta^2$-torsion, then
\begin{align*}
    \rho_n \colon H^n(X, \mathbf{I}^n) \to \Ch^n(X)
\end{align*}
is injective on the image of $\beta_{n-1}$, and in particular $\ker(\Sq^2_{n-1}) = \ker(\beta_{n-1}) = \im(\rho_{n-1})$.
\end{lemma}

\begin{proof}
  Suppose we have an $x\in \Ch^{n-1}(X)$ so that $\rho_n \left( \beta_{n-1}(x) \right) = 0$. We'd like to argue that $\beta_{n-1}(x) = 0$. Since $\rho_n \left( \beta_{n-1}(x) \right)=0$, this implies by the B\"ar sequence that $\beta_{n-1}(x) = \eta y$ for some $y\in H^n(X,\mathbf{I}^{n+1})$. Since $\eta y \in \im(\beta_{n-1})$, we have that $\eta y$ is $\eta$-torsion, implying $\eta^2 y = 0$ (in $H^n(X,{\bf I}^{n-1})\cong H^n(X,{\bf W})$). By assumption of $\eta^2$-torsion-freeness, we conclude that $\eta y$ and/or $y$ are actually zero, in either case this tells us $\beta_{n-1}(x) = 0$, and we have the first conclusion.

Since $\rho_n$ is injective on the image of $\beta_{n-1}$, the kernels of $\beta_{n-1}$ and $\Sq^2_{n-1} = \rho_n \circ \beta_{n-1}$ agree, and the B\"ar sequence identifies $\ker(\beta_{n-1}) = \im(\rho_{n-1})$.
\end{proof}

As a particular example, the assumption of $\eta^2$-torsionfreeness is satisfied if $H^\bullet(X, {\bf I}^{\bullet-1})\cong H^\bullet(X,{\bf W})$ is free as a ${\rm W}(k)$-module, since in this case also $H^n(X,{\bf I}^{n})$ has no $\eta$-power torsion, and \autoref{lem:eta-squared-torsion} applies.

\subsection{Twists and orientations}

Given a line bundle $\mathcal{L} \to X$, we can twist an $\A^1$-invariant sheaf of abelian groups by the bundle in order to obtain \textit{twisted cohomology}. The different twists that can appear are captured by the Picard group $\Pic(X)$ of the scheme, but this can change depending on the orientation data attached to the cohomology. Oriented sheaves, for example Milnor $K$-theory, do not see twists. Quadratically oriented theories, like Milnor--Witt $K$-theory, Witt theory, and ${\bf I}^j$-cohomology, are insensitive to twists by squares of line bundles, hence the possible twists are indexed by $\Pic(X)/2$.

\begin{example}\label{ex:twists-bgln} We have that $\Ch^1(\BGL_n) = \Z/2\Z$, so there are two twists given by the trivial bundle $\O_{\BGL_n}$ and the tautological bundle $\O_{\BGL_n}(-1)$. As a particular case when $n=1$, we have two twists over $\P^\infty$ as well. This compares well to the real realization where there are two isomorphism classes of rank 1 local systems over $\mathbb{RP}^\infty$, or more generally ${\rm BO}(n)$. 
\end{example}

\subsection{The localization sequence}

Suppose $Z \hookto X$ is a closed immersion of smooth $k$-schemes, and $U = X-Z$ is the open complement. Let $\mathscr{F}$ be a strongly $\A^1$-invariant sheaf of abelian groups on $X$, and let $\mathcal{L} \to X$ be a line bundle. Then we obtain a long exact sequence on cohomology with supports:
\begin{equation}\label{eqn:localization}
\begin{aligned}
    \cdots \to H^j_Z(X, \mathscr{F} (\mathcal{L})) \xto{i_\ast} H^{j}(X, \mathscr{F} (\mathcal{L})) \to H^{j}(U, \mathscr{F} ( \mathcal{L})) \xto{\del} \cdots
\end{aligned}
\end{equation}
In case the sheaf $\mathscr{F}$ is part of a homotopy module $M_\bullet$, then purity allows to identify the cohomology with supports as cohomology of $Z$:
\begin{align*}
  H^j_Z(X,M_q(\mathcal{L}))\cong H^{j-d}(Z,M_{q-d}(\mathcal{L}\otimes\det\nolimits^{-1}N_{Z/X}))
\end{align*}
Here $N_{Z/X}$ denotes the normal bundle of the immersion $Z\hookto X$, and $d$ denotes the codimension of the immersion. Cohomology with supports can be interpreted as the cohomology of the Thom space ${\rm Th}(N_{Z/X})$, and essentially the twists are a different way to talk about cohomology of Thom spaces of line bundles.

\begin{example}\label{ex:chow-localization} When $\mathscr{F} = \mathbf{K}_j^{\rm M}$, this specializes to the well-known localization sequence on Chow groups 
\begin{align*}
    \CH^{i-d}(Z) \to \CH^i(X) \to \CH^i(U) \to 0,
\end{align*}
where $d = \codim(Z,X)$. Note that there is no twist in the Chow group of $Z$ because Chow groups are part of a GL-orientable cohomology theory.
\end{example}

\begin{example}\label{ex:i-star-zero-section}
  Let $i\colon Z \to X$ be the inclusion of the zero locus of a vector bundle over $X$. Assume that $M_\bullet$ is an ${\rm SL}^c$-orientable homotopy module, so that $H^\bullet(-,M_\bullet)$ supports a theory of Euler classes. Then $i_\ast$ can be interpreted as multiplication with the Euler class of the bundle. This is a classical fact for Chow groups, e.g. \cite[Example~3.3.2]{Fulton}.
\end{example}

\subsection{K\"{u}nneth formulas}

Given a group scheme $G$, one wants to construct a motivic classifying space ${\rm B}G$. Analogous to classical topology, we'd like to define this to be the quotient ${\rm E}G/G$ of a contractible space with a free $G$-action. In the algebraic setting, there are different topologies (Zariski, Nisnevich and \'etale, to mention the most relevant ones for our setting), and consequently different possible such quotients (in Zariski, Nisnevich or \'etale sheaves).

While the Zariski and Nisnevich classifying spaces typically agree, the \'etale topology is typically different. For the Zariski and Nisnevich classifying spaces, one can work with the simplicial sheaf bar construction model, while the classifying spaces for the \'etale topology have Ind-scheme models and can be approximated by smooth schemes. In particular, Totaro \cite{Totaro} has used the latter approach in his definition and computations of Chow groups of classifying spaces. His work shows that in order to do cohomology computations for ${\rm B}_\et G$ (for cohomology theories in the heart of the homotopy $t$-structure), it suffices to work with a finite-dimensional model for ${\rm B}_\et G$, with the appropriate dimension depending on the cohomological degrees of interest.

Explicitly, let $V$ be a $G$-representation, and suppose that the locus $Z \subseteq V$ on which $G$ does not act freely is of sufficiently high codimension. Then the quotient space $(V-Z)/G$ provides a model for ${\rm B}G$ up to a certain dimension, roughly equal to the codimension of $Z$. The following proposition is well-known, and states that nice approximations always exist, cf.~\cite[Remark~1.4]{Totaro}:

\begin{proposition} If $G$ is an affine finite type algebraic group scheme over a field $k$, then there always exist $G$-representations $V$ whose nonfree locus $S \subseteq V$ is of arbitrarily high codimension. In particular such a $G$ admits a well-defined ${\rm B}_\et G$ in the homotopy category of motivic spaces $\mathcal{H}(k)$.
\end{proposition}

%\begin{proof} Any affine finite type group admits a faithful representation in $\GL_n$, so it suffices to argue this holds for linear algebraic groups, which follows for the reason outlined in \cite[Remark~1.4]{Totaro}.
%\end{proof}

Totaro's definition is originally stated for Chow groups of classifying spaces, but it holds in a more general context, which we now outline. Recall that a \textit{homotopy module} $\mathbf{M}_\bullet$ is a strictly $\A^1$-invariant $\Z$-graded Nisnevich sheaf of abelian groups equipped with a desuspension isomorphism $\mathbf{M}_n \xto{\sim} (\mathbf{M}_{n+1})_{-1}$ for every $n$. It is a classical fact that $\pi_0 E$ is a homotopy module for any $E\in \SH(k)$, and in fact $\pi_0$ induces an equivalence between the heart of the homotopy $t$-structure on $\SH(k)$ with the category of homotopy modules, with the inverse given by Eilenberg--Mac~Lane spectra for homotopy modules.

The following now allows to compute (or, depending on ones point of view, define) the cohomology of classifying spaces in terms of smooth approximations of the classifying space, and in the generality below was established by di~Lorenzo and Mantovani in \cite[Proposition~2.2.10]{equivariant-CHW}:

\begin{proposition} If $Z \subseteq V$ has codimension $i$, then we have that
\begin{align*}
    H^\bullet \left( {\rm B}G, \mathbf{M} \right) \cong H^\bullet \left( (V-Z)/G, \mathbf{M} \right),
\end{align*}
for $\bullet \le i$.
\end{proposition}

Totaro is able to leverage this to prove a K\"{u}nneth formula for the Chow groups of classifying spaces:

\begin{theorem}\label{thm:chow-kunneth} \cite[\S6]{Totaro} Working over a field, assume that $G_1$ and $G_2$ are (finite type) algebraic groups such that the classifying space ${\rm B}G_1$ has enough linear approximations (in the sense of \cite{Totaro-motive}). Then there is a K\"{u}nneth isomorphism
\begin{align*}
    \CH^\bullet({\rm B}G_1 \times {\rm B}G_2) \cong \CH^\bullet({\rm B}G_1) \otimes_\Z \CH^\bullet({\rm B}G_2).
\end{align*}
\end{theorem}

We should comment a bit about what goes into this proof and into the assumptions. If $X$ is a cellular variety, then its Chow groups are free abelian indexed over the cells. In particular, if we look at the Chow group localization sequence arising from the inclusion of a cell and its complement, we obtain a short exact sequence of free abelian groups. Hence tensoring with $- \otimes_\Z \CH^\bullet(Y)$ for any $Y$ will still be exact. This allows us to prove a K\"{u}nneth theorem if one of the varieties is cellular, by inducting on the codimension of the cells in its stratification \cite[\S6]{Totaro}.

This hinges on an inductive idea --- if $Z$ and $X\setminus Z$ satisfy a K\"{u}nneth theorem, then via localization $X$ will as well. Thus we might expect that a large class of varieties which satisfy a K\"{u}nneth theorem might be one which includes affine space, and is closed under some restricted two-out-of-three property. Indeed this is roughly the definition of a \textit{linear scheme}, a definition due to Janssen \cite{Jannsen} and under slightly different assumptions Totaro \cite{Totaro-motive}. When ${\rm B}G_1$ admits enough linear models, these linear models can be used to provide a K\"{u}nneth theorem in the degrees where they are effective at approximating ${\rm B}G_1$. More generally, one also has a K\"{u}nneth isomorphism for (higher) Chow groups coming from cellularity of varieties, as in \cite[\S6]{Krishna}. 

\begin{remark} In constructing a model for ${\rm B}G$, we want a highly connected variety on which $G$ acts freely, as finite-dimensional approximation modeling ${\rm B}G = {\rm E}G/G$. It is not strictly necessary to start with a $G$-representation $V$, but representations are an (easily accessible and well-understood) source of contractible varieties with $G$-action. The complement $V\setminus S$ is then highly connected whenever $S$ has high codimension, and hence provides a good model for ${\rm E}G$.

  As an example, we claim that finite Grassmannians provide linear models for infinite ones. Indeed, we could take a Stiefel variety $\GL_{n+k}/\GL_k$, on which $\GL_n$ acts freely. This Stiefel variety is highly connected, and after quotienting by the $\GL_n$-action, we obtain a Jouanalou device over the Grassmannian, so it is $\A^1$-equivalent to $\Gr(n,k)$. This confirms what we might suspect, that $\Gr(n,k)$ is a model for $\BGL_n$. Moreover, the Schubert cell stratification of the Grassmannian shows that it is linear (in either the sense of Jannsen or Totaro).
\end{remark}

Using similar techniques, Hudson, Matszangosz and Wendt obtained a K\"{u}nneth theorem for Witt-sheaf cohomology groups, cf.~\cite[Proposition~4.7]{HMW}. Note that the cellularity required for this K\"unneth formula is stronger than the notions of Jannsen and Totaro, and actually requires a stratification by affine spaces. 

\begin{theorem} 
  Let $X_1,X_2\in\Sm_k$ for $k$ perfect of characteristic $\ne 2$, let $\mathcal{L}_i \to X_i$ be line bundles, suppose that $X_1$ is cellular, and suppose that all $H^q(X_i, \mathbf{W}(\mathcal{L}_i))$ are free ${\rm W}(k)$-modules for $i=1,2$. Then we have a K\"{u}nneth isomorphism given by the exterior product map
  \begin{align*}
    H^\bullet \left( X_1 \times X_2, \mathbf{W} \left( \mathcal{L}_1 \sqtimes \mathcal{L}_2 \right) \right) \cong H^\bullet \left( X_1, \mathbf{W}(\mathcal{L}_1) \right) \otimes_{{\rm W}(k)} H^\bullet \left( X_2, \mathbf{W}(\mathcal{L}_2) \right).
  \end{align*}
\end{theorem}

We can leverage this, together with Totaro's work, to prove a K\"{u}nneth isomorphism for infinite Grassmannians.

\begin{example}\label{ex:witt-kunneth-grassmannians} For any $m,n\ge 1$, and any line bundles $\mathcal{L} \to \BGL_m$ and $\mathcal{L}' \to \BGL_n$, we have a K\"{u}nneth isomorphism
\begin{align*}
    H^\bullet \left( \BGL_m \times \BGL_n, \mathbf{W} \left( \mathcal{L} \sqtimes \mathcal{L}'\right) \right) \cong H^\bullet \left( \BGL_m, \mathbf{W}(\mathcal{L}) \right) \otimes_{{\rm W}(k)} H^\bullet \left( \BGL_n, \mathbf{W}(\mathcal{L}') \right).
\end{align*}
\end{example}
\begin{proof} Suppose we want to restrict to proving this in the range $\bullet \le s$, and then we want to let $s$ tend to infinity. In such a finite range, since the Witt sheaf extends to a homotopy module $(\mathbf{W})_n$, we can reduce to geometric models of these classifying spaces, which can be chosen to be finite Grassmannians. The Witt-sheaf cohomology groups of finite Grassmannians are free ${\rm W}(k)$-modules by the computations in \cite{Wendt-Gr}, and the finite Grassmannians are cellular, so we can apply the K\"{u}nneth isomorphism for Witt cohomology to conclude.
\end{proof}

\section{The classifying space of quadratically oriented bundles}\label{sec:BSLnc}

In this section we construct the classifying space $\BSL_n^c$ for quadratically oriented rank $n$ bundles, and provide a model for it as an ind-variety.

\subsection{Quadratically oriented vector bundles}

A \textit{quadratic orientation} (called an \textit{orientation} in \cite[Definition~4.3]{Morel}) on a (topological or algebraic) vector bundle $E \to X$ is a choice of isomorphism $\rho \colon \det E \simeq \Theta^{\otimes 2}$, where $\Theta$ is a line bundle on $X$. Phrased differently,  a quadratic orientation is a choice of square root of the determinant bundle.

We remark that quadratically oriented bundles can have different quadratic orientations, corresponding to different choices of square root of the determinant bundle. The different choices are a torsor under the group ${}_2\Pic(X)$ of 2-torsion line bundles on $X$, which parametrizes the square roots of the trivial bundle.

\begin{remark} $\ $
\begin{enumerate}
\item An \textit{oriented} bundle, i.e., one whose determinant is a trivial line bundle, is canonically quadratically oriented.
\item A rank $n$ vector bundle $E \to X$ over a smooth $n$-dimensional base is \textit{relatively oriented} if $\Hom(\det {\rm T}X, \det E)$ is quadratically oriented. If $n$ is odd this is the same as asking $\Hom({\rm T}X,E)$ to be quadratically oriented. Relative orientations play a key role both in real enumerative geometry, cf.~\cite{OkTel}, as well as in $\A^1$-enumerative geometry, cf.~\cite{EulerBW3}.
\item Note that a relatively oriented bundle need not be quadratically oriented, for example the tangent bundle of any smooth variety admits a canonical relative orientation, but the tangent bundle on $\P^{2n}$, for instance, is not quadratically oriented. 
    \item Similarly, there exist quadratically oriented bundles which are not relatively oriented, for instance $\O_{\P^2}(2)^{\oplus 2}$.
\end{enumerate}
\end{remark}

\begin{definition}\label{def:metalinear-grp} \cite[Remark~2.8]{Ananyevskiy} We define the \textit{metalinear group}\footnote{This terminology is taken from \cite[\S 3.3.2]{AHW2}.} $\SL_n^c$ to be the kernel of the homomorphism
\begin{align*}
    \GL_n \times \mathbb{G}_{\rm m} &\to \mathbb{G}_{\rm m} \\
    (g,t) &\mapsto t^{-2}\det(g).
\end{align*}
\end{definition}

In particular $\SL_n^c$-torsors on $X$ are quadratically oriented vector bundles, in the sense described above. 

\subsection{Metalinear Hilbert 90}

The following proposition establishes that the group scheme $\SL_n^c$ is \textit{special}, a result which we will use to discuss its classifying space. We don't know of an explicit result in the literature stating that ${\rm SL}_n^c$ is special, but it is likely clear to anyone who ever considered the question. For example, Ananyevskiy explicitly only considers Zariski-locally trivial ${\rm SL}_n^c$-torsors in \cite{Ananyevskiy}.

\begin{proposition}
\label{prop:slnc-special}
The group ${\rm SL}_n^c$ is special in the sense of Serre, i.e., the natural change-of-topology map is a bijection:
  \begin{align*}
    {\rm H}^1_{\rm Zar}(X,{\rm SL}_n^c)\xrightarrow{\cong} {\rm H}^1_\et(X,{\rm SL}_n^c)
  \end{align*}
\end{proposition}

\begin{proof}
  We use the defining short exact sequence 
  \begin{align*}
    0 \to \SL_n^c \to \GL_n \times \mathbb{G}_{\rm m} \xto{\det \times (-)^{-2}} \mathbb{G}_{\rm m} \to 0,
  \end{align*}
  of algebraic groups, which gives rise to an exact sequence of non-abelian \'etale cohomologies
  \[
    {\rm H}^0(X,{\rm GL}_n\times\mathbb{G}_{\rm m})\xrightarrow{\det\times(-)^{-2}} {\rm H}^0(X,\mathbb{G}_{\rm m})\to {\rm H}^1_\et (X,{\rm SL}_n^c)\to {\rm H}^1_\et (X,{\rm GL}_n\times\mathbb{G}_{\rm m})
  \]
  Note that the first cohomologies here are only pointed sets, and exactness at ${\rm H}^1_\et(X,{\rm SL}_n^c)$ only means that the image of the first map $\det\times(-)^{-2}$ equals the preimage of the base-point in ${\rm H}^1_\et(X,{\rm GL}_n\times\mathbb{G}_{\rm m})$. It suffices to show the vanishing of ${\rm H}^1_\et(X,{\rm SL}_n^c)$ for $X={\rm Spec}(R)$ the spectrum of a local ring $R$. In this case, ${\rm H}^1_\et(X,{\rm GL}_n\times\mathbb{G}_{\rm m})=\{\ast\}$ since ${\rm GL}_n\times\mathbb{G}_{\rm m}$ is well-known to be special (Hilbert 90). It remains to see that $\det\times(-)^{-2}$ is surjective, but that is easy to see since any unit in $R$ can be realized as determinant of a matrix.
\end{proof}

\begin{remark}
    Alternatively, we could use the Huruguen--Merkurjev theorem on classification of reductive special groups, cf. \cite{merkurjev:special}. 
\end{remark}

\begin{corollary}\label{cor:bslnc} We have that the following hold:
\begin{enumerate}
    \item The natural change-of-topology maps induce equivalences
    \begin{align*}
        {\rm B}_\Zar \SL_n^c \simeq {\rm B}_{\rm Nis} \SL_n^c \simeq {\rm B}_\et\SL_n^c,
    \end{align*}
    hence we can unambiguously write $\BSL_n^c$ for any of these.

    \item The classifying space $\BSL_n^c$ fits into a pullback diagram of motivic spaces
    \[ \begin{tikzcd}
    \BSL_n^c \rar\dar\pb & \BGm\dar["(-)^2" right]\\
    \BGL_n\rar["\Bdet" below] & \BGm.
    \end{tikzcd} \]
\end{enumerate}
\end{corollary}
\begin{proof} The first statement is an immediate consequence of the fact that $\SL_n^c$ is special as in \autoref{prop:slnc-special}. For the second statement, we use that the classifying space functor converts fiber sequences of group schemes to $\A^1$-fiber sequences of motivic spaces (cf. the discussions around homogeneous space fiber sequences in \cite{AHW2}) and is product-preserving.
\end{proof}

\subsection{The classifying space \texorpdfstring{$\BSL_n^c$}{BSLnc} as an ind-variety}\label{subsec:ind-variety}

The motivic space $\BGL_n$, while not a variety, can be naturally modeled as an ind-variety, or formal colimit of varieties $\colim_{m\to \infty} \Gr(n,m)$ \cite[Proposition~4.3.7]{MV99}. Similarly, $\BSL_n$ can be modeled as a colimit of the (non-projective) oriented Grassmannians $\colim_{m\to\infty}\widetilde{\Gr}(n,m)$. With this in mind, one may ask whether an analogous statement is true for the classifying space $\BSL_n^c$, and we will discuss this in what follows. To do this, we will in particular want to describe scheme models of the maps appearing in the pullback description of $\BSL_n^c$ in \autoref{cor:bslnc}.  %We can see this from two perspectives, one being the definition we have already provided, and another coming from an alternative way to view $\BSL_n^c$.

We begin with the observation that the determinant map $\BGL_n \to \P^\infty$ admits a very natural model in the world of varieties.

\begin{proposition} The determinant map $\BGL_n \to \P^\infty$ is the colimit of the Pl\"{u}cker embeddings $\Gr(n,m) \to \P^{\binom{m+n}{n}-1}$ as $m$ tends to $\infty$.
\end{proposition}
\begin{proof} By construction, the hyperplane class $\O(1)$ on projective space pulls back to the top wedge power $\wedge^r \mathcal{S}^\ast$ of the dual tautological bundle over the Grassmannian. Taking duals, we see that the tautological bundle $\O(-1)$ pulls back along the Pl\"{u}cker embedding to the determinant bundle $\wedge^r \mathcal{S}$.

  Note that the Pl\"ucker embeddings are compatible with the stabilization maps on the Grassmannian side. For given $m$, let $V$ be an $(m+n)$-dimensional $k$-vector space, and the Pl\"ucker embedding is given by
  \begin{align*}
    \Gr(n,m)&\to\P^{\binom{m+n}{n}-1} \\
    {\rm span} \{v_1,\dots,v_n\} &\mapsto v_1\wedge\cdots\wedge v_n.
  \end{align*}
  On the Grassmannian side, the stabilization $\Gr(n,m)\to\Gr(n,m+1)$ is induced by the embedding $V=V\oplus\{0\}\hookrightarrow V\oplus k$, with last coordinate zero. This embedding induces an embedding of projective spaces
  \begin{align*}\mathbb{P}\left(\bigwedge\nolimits^n V\right)\hookrightarrow\mathbb{P}\left(\bigwedge\nolimits^n(V\oplus k)\right)
  \end{align*}
  whose image consists of $n$-fold wedges of basis vectors whose last coordinate is zero, and which is a linear subspace of codimension $\dim_k\bigwedge\nolimits^{n-1}V=\binom{m+n}{n-1}$. Consequently, we get a commutative diagram
  \[ \begin{tikzcd}
    {\Gr(n,m)}\rar\dar["\text{Pl}" left] & {\Gr(n,m+1)}\dar["\text{Pl}" right]\\
    \P\left(\bigwedge\nolimits^n V\right)\rar & \P\left(\bigwedge\nolimits^n(V\oplus k)\right),
  \end{tikzcd} \]
  and the colimit of the vertical maps as $m\to \infty$ realizes the determinant map. 
\end{proof}

Unfortunately, the squaring map ${\rm B}\mathbb{G}_{\rm m}\to {\rm B}\mathbb{G}_{\rm m}$ is not representable (in the sense usually used for stacks): for a morphism $X\to{\rm B}\mathbb{G}_{\rm m}$, the source of the base-changed squaring map ${\rm B}\mathbb{G}_{\rm m}\times_{{\rm B}\mathbb{G}_{\rm m}} X\to X$ isn't a scheme. Another way to phrase the problem is that the homotopy fiber of the squaring map is ${\rm B}_\et\mu_2$, which is not a scheme.

Nevertheless, we can give an explicit description of $\BSL_n^c$ as an ind-smooth ind-scheme. This is based on viewing $\BSL_n^c$ as the complement of the zero section of a line bundle over $\BGL_n \times \P^\infty$. In \autoref{cor:bslnc}, we saw that $\BSL_n^c$ can be written as fiber product; this is equivalent to writing $\BSL_n^c$ as homotopy fiber of the morphism
\[
  {\rm B}(\det\times(-)^{-2})\colon \BGL_n\times\P^\infty\to\P^\infty
\]
which for now we'll just call $f$ for simplicity. We may then view $\BSL_n^c$ as the total space of the $\mathbb{G}_{\rm m}$-torsor over $\BGL_n\times \mathbb{P}^\infty$ classified by $f$. Explicitly, there is a pullback square of the form
\[ \begin{tikzcd}
    \BSL_n^c\rar\dar\pb & {\rm E} \mathbb{G}_{\rm m} = \A^\infty\minus\{0\}\dar\\
    \BGL_n \times \P^\infty\rar["f" below] & {\rm B} \mathbb{G}_{\rm m} = \P^\infty,
\end{tikzcd} \]
where the projection $\A^\infty\minus\{0\} \to \P^\infty$ quotients out by the diagonal action of $\mathbb{G}_{\rm m}$. Thinking about each fiber as living inside an affine line, we can think about this as the complement of the zero section of the tautological line bundle over $\P^\infty$. That is, there is a pullback square
\[ \begin{tikzcd}
    f^\ast \O(-1)\rar\dar\pb & \O(-1)\dar\\
    \BGL_n \times \P^\infty \rar["f" below] & \BGm
\end{tikzcd} \]
such that the pullback square describing $\BSL_n^c$ is obtained by taking complements of zero sections. From this we can conclude that $\BSL_n^c$ is the complement of the image of the zero section $z \in \Gamma(\BGL_n \times \P^\infty, f^\ast \O(-1))$. Note that the line bundle $f^\ast\O(-1)$ classified by $f$ could alternatively be written as $\det\boxtimes\O(2)$.

With this, we can now describe $\BSL_n^c$ as an ind-scheme. We can first write $\BGL_n\times\P^\infty$ as colimit of the smooth projective schemes $\Gr(n,m)\times\P^N$. Over each of these, we have a bundle $\det\boxtimes\O(2)$, obtained by pulling back the corresponding bundle $\det\boxtimes\O(2)$ along the inclusion $\Gr(n,m)\times\P^N\to\BGL_n\times\P^\infty$. If we denote the complement of the zero section of $\det\boxtimes\O(2)$ over $\Gr(n,m)\times\P^N$ by $\Gr^c(n,m;N)$, we get the following
\begin{corollary}\label{cor:BSLnc-metalinear-grassmannian} 
We have that
\begin{align*}
    \BSL_n^c \simeq \colim_{m,N\to\infty} \Gr^c(n,m;N).
\end{align*}
\end{corollary}
\begin{proof} Colimits of motivic spaces are universal, meaning they commute with pullbacks. Hence we can model $\BSL_n^c$ as a colimit of the $\Gr^c(n,m;N)$'s.
\end{proof}

\begin{remark}
  \label{rem:metalinear-Gr}
  We will call the colimit
  \[
  \Gr^c(n,\infty):=\colim_{m,N\to\infty}\Gr^c(n,m;N)
  \]
  the \emph{infinite metalinear Grassmannian}. We could also call the finite-dimensional schemes $\Gr^c(n,m;N)$ metalinear Grassmannians, but they don't quite have the look-and-feel of Grassmannians, due to the additional projective space appearing in the definition. It seems that it is impossible to get rid of this additional factor, which is related to the squaring map ${\rm B}\mathbb{G}_{\rm m}\to{\rm B}\mathbb{G}_{\rm m}$ not being representable, as discussed above. In particular, if we want to define metalinear Grassmannians as fiber product of the Pl\"ucker embedding $\Gr(n,m)\to \P^N$ (as a model of the determinant map) and a model $X\to \P^N$ of the squaring map, this seems to always introduce the additional $\P^N$-factor one way or another.
\end{remark}

The above description of $\BSL_n^c$ as complement of the zero section of the line bundle $\det\boxtimes\O(2)$ over $\BGL_n\times\P^\infty$ will be important for a number of arguments in the remainder of the paper. It provides a localization sequence associated to
\begin{align*}
    \im(z) \clhookto f^\ast \O(-1) \ohookfrom f^\ast \O(-1) \minus\im(z),
\end{align*}
where $z$ is the zero section, whose image is $\BGL_n \times \P^\infty$. Similarly we may contract the fibers of the line bundle to see that the total space is also $f^\ast\O(-1) \simeq \BGL_n \times \P^\infty$. The complement is equivalent to $f^\ast \O(-1)\minus\im(z)\simeq \BSL_n^c$ as we have argued above. We will use the associated localization sequence to deduce information about various cohomology theories of $\BSL_n^c$.

\subsection{Cellularity and dualizability}

The suspension spectra of Grassmannians are well-known to be cellular and strongly dualizable. This fact is important in a number of places, for example in the computation of homology of the Thom spectrum MGL. An analogous result for oriented Grassmannians was proved in \cite[Lemma~4.15]{bachmann:hopkins}. We observe that the same argument can also be used to show that the metalinear Grassmannians enjoy the same properties: 

\begin{proposition}
  \label{prop:cellular-dualizable}
  The suspension spectrum $\Sigma^\infty_+\Gr^c(n,m;N)$ of the metalinear Grassmannian of \autoref{rem:metalinear-Gr} is cellular and strongly dualizable.
\end{proposition}

\begin{proof}
  The proof follows the arguments in \cite[Lemma~4.15]{bachmann:hopkins}.

  To show strong dualizability, we use the description of $\Gr^c(n,m;N)$ as complement of the zero section of $f^\ast\mathcal{O}_{\P^\infty}(-1)\cong\det\boxtimes\O(2)$ on $\Gr(n,m)\times\P^N$. We consequently get a cofiber sequence (of suspension spectra): 
  \begin{align*}
    \Gr^c(n,m;N)\to\Gr(n,m)\times\mathbb{P}^N\to\Th(\det\boxtimes\O(2))
  \end{align*}
  The strong dualizability of $\Gr^c(n,m;N)$ then follows from the well-known strong dualizability of Grassmannians $\Gr(n,m)$, projective spaces $\P^N$ and Thom spaces over these.

  To show cellularity, we can use that $\Gr(n,m)\times\P^N$ has a well-known cell structure: it is a projective homogeneous variety for $\GL_{n+m}\times\GL_N$, and the cells are obtained as orbits of a Borel subgroup. This can be turned into an unstable cell structure as described in \cite[Section~3.3]{cells}. Realizing $\Gr^c(n,m;N)$ as complement of the zero-section of a line bundle over $\Gr(n,m)\times\P^N$, we can lift the cell structure to $\Gr^c(n,m;N)$: over each cell of $\Gr(n,m)\times\P^N$, the line bundle trivializes and therefore we have a stratification of $\Gr^c(n,m;N)$ by cells of the form $\mathbb{A}^d\times\mathbb{G}_{\rm m}$ (where the dimension $d$ of the affine space will depend on the cell). This can be turned into an unstable (or stable) cell structure, where cell attachments happen via cofiber sequences $X\setminus X_i\to X\setminus X_{i-1}\to\Th(N_i)$ with $\Th(N_i)$ a wedge of spheres ${\rm S}^{2n-1,n}$. Inductively, this would also be an alternative way of seeing the strong dualizability.
\end{proof}

\begin{remark}
  At this point it is not quite clear if we can write the metalinear Grassmannians as homogeneous spaces under $\SL_n^c$. If possible, the cellularity could also be deduced from the Bruhat decomposition for reductive groups, as in \cite[Lemma~4.15]{bachmann:hopkins}.
\end{remark}

\section{Oriented cohomologies of \texorpdfstring{$\BSL_n^c$}{BSLnc}}\label{sec:oriented-theories}

In this section we compute the cohomology of $\BSL_n^c$ in various $\GL$-oriented cohomology theories, namely Chow groups and motivic cohomology. We further investigate the action of Steenrod squares on the mod two Chow groups of $\BSL_n^c$, which will be needed in \autoref{sec:chow-witt} to compute the Chow--Witt groups of $\BSL_n^c$.

\subsection{Chow rings and possible twists on \texorpdfstring{$\BSL_n^c$}{BSLnc}}

In this section we compute the Chow rings for $\BSL_n^c$. These will be needed both as piece of the Chow--Witt computation as well as to understand the possible twists in  $\Pic(\BSL_n^c)/2$ that may appear in the Chow--Witt groups of $\BSL_n^c$.

Via the discussion in \autoref{subsec:ind-variety}, we have a localization sequence associated to
\begin{equation}\label{eqn:localization-seq-bslnc}
\begin{aligned}
    \BGL_n \times \P^\infty \clhookto \BGL_n \times \P^\infty \ohookfrom \BSL_n^c.
\end{aligned}
\end{equation}

\begin{proposition} For any $i$, there is a four-term localization sequence
\begin{align*}
    \CH^i(\BGL_n \times \P^\infty) \xto{-\cdot c_1(f^\ast\O(-1))} \CH^{i+1}(\BGL_n \times \P^\infty) \to \CH^{i+1}(\BSL_n^c) \to 0,
\end{align*}
where the first map is multiplication by $c_1 \left( f^\ast \O(-1) \right)$ by \autoref{ex:i-star-zero-section}.
\end{proposition}

We'd like to leverage this to compute the Chow groups of $\BSL_n^c$. First note that we can apply the K\"{u}nneth formula for Chow groups (\autoref{thm:chow-kunneth}) to obtain the following computation.

\begin{corollary} We have that
\begin{align*}
    \CH^\bullet \left( \BGL_n \times \P^\infty \right) \cong \Z[c_1, \ldots, c_n, \theta]
\end{align*}
with the class $\theta=c_1(\O_{\mathbb{P}^\infty}(-1))$ of degree $|\theta|=1$, and $|c_i| = i$. 
\end{corollary}

We see then that $\CH^\bullet(\BSL_n^c)$ is the cokernel of multiplication by the first Chern class of $f^\ast \O(-1)$, which we compute as follows:

\begin{proposition} The first Chern class of the pullback of $\O(-1)$ is
\begin{align*}
    c_1 \left( f^\ast \left( \O(-1) \right) \right) = c_1 - 2\theta.
\end{align*}
\end{proposition}
\begin{proof} We see that $f^\ast \O(-1)$ will be the external tensor product of the pullback of bundles to each of $\BGL_n$ and $\P^\infty$. This becomes the tensor product of line bundles, which will translate to addition on $\CH^1$. For the determinant map
\begin{align*}
    \Bdet\colon \BGL_n \to \P^\infty,
\end{align*}
the pullback of $c_1$ will be $c_1$, since $c_1(\det E) = c_1(E)$ for any vector bundle $E$. The squaring map on $\P^\infty$ will pull back $\O(-1)$ to $\O(-2)$, and then inverting it will send it to $\O(2)$. Since $\O(-1)$ corresponds to $\theta$, we have that $\O(2)$ corresponds to $-2\theta$.
\end{proof}

\begin{corollary}\label{cor:chow-product} We have that
\begin{align*}
    \CH^\bullet \left( \BSL_n^c \right) &= \frac{\Z[c_1, \ldots, c_n,\theta]}{\left\langle c_1 - 2\theta \right\rangle} \\
    \Ch^\bullet(\BSL_n^c) &= \Z[\bar{c}_2, \ldots, \bar{c}_n,\theta]
\end{align*}
\end{corollary}

From this discussion we see that
\begin{align*}
    \CH^1 \left( \BGL_n \times \P^\infty \right) \cong \Z \times \Z,
\end{align*}
where the copies of $\Z$ are generated by $\theta$ and $c_1$. As possible twists are determined by the mod 2 Picard group, we can mod out above to get
\begin{align*}
    \Ch^1 \left( \BGL_n \times \P^\infty \right) \cong \Z/2\Z \times \Z/2\Z,
\end{align*}
where one factor comes from the determinant of the universal bundle on $\BGL_n$ and the other factor comes from the determinant of the universal bundle on $\P^\infty$. We can write the four possible twists as
\begin{align*}
    &\O_{\BGL_n \times \P^\infty},\quad \quad &\O_{\BGL_n} \sqtimes \O_{\P^\infty}(-1),\\
    &\O_{\BGL_n}(-1)\sqtimes \O_{\P^\infty}, \quad \quad & \O_{\BGL_n}(-1) \sqtimes \O_{\P^\infty}(-1).
\end{align*}
Since the map
\begin{align*}
    \CH^1 \left( \BGL_n \times \P^\infty \right) \to \CH^1 \left( \BSL_n^c \right)
\end{align*}
is the quotient by $c_1 - 2\theta$, we get two possible twists for $\BSL_n^c$, namely $\O_{\P^\infty}(-1)$ and the trivial one.

To connect this to the picture of quadratically oriented vector bundles, note that the bundle we denoted by $\O_{\P^\infty}(-1)$ above is the bundle $\Theta$ providing the quadratic orientation. Restricting $f^\ast\mathcal{O}(-1)$ to the complement of the zero section forces $\mathcal{O}_{\BGL_n}(-1)\boxtimes\mathcal{O}_{\P^\infty}$ to be isomorphic to $\mathcal{O}_{\BGL_n}\boxtimes\mathcal{O}_{\P^\infty}(-2)$, i.e., $\Theta^2\cong\det$ on $\BSL_n^c$ which on the level of the Picard group is encoded in $c_1-2\theta$.

\subsection{The Steenrod square action}

In this section we characterize how $\Sq^2$ acts on the mod 2 Chow groups of $\BSL_n^c$. This helps us to understand the image of the Bockstein homomorphism, which will in turn allow us to compute the ${\bf I}^j$-cohomology of $\BSL_n^c$ as well as its Chow--Witt theory.

We note that the maps $\BSL_n \to \BSL_n^c \to \BGL_n$ induce pullback maps on $\Ch^\bullet(-)$ which are compatible with the Steenrod algebra structure --- that is, the pullbacks are morphisms of modules over the Steenrod algebra. In particular, we have a commutative diagram
\[ \begin{tikzcd}
    \Ch^j(\BGL_n)\rar\dar["\Sq^2"] & \Ch^j(\BSL_n^c)\dar["\Sq^2"]\rar & \Ch^j(\BSL_n)\dar["\Sq^2"]\\
    \Ch^{j+1}(\BGL_n)\rar & \Ch^{j+1}(\BSL_n^c)\rar & \Ch^{j+1}(\BSL_n).
\end{tikzcd} \]
We can use this to understand how $\Sq^2$ acts on $\Ch^\bullet(\BSL_n^c)$. Via \cite[Remark~10.5]{Fasel-IJ} or \cite[p.~947]{HornbostelWendt}, we have that
\begin{align*}
    \Sq^2(\bar{c}_{2i})=\bar{c}_{2i+1}
\end{align*}
in $\Ch^\bullet(\BSL_n)$. From \cite[Proposition~3.12]{Wendt-Gr} we have that $\Sq^2$ acts on $\Ch^\bullet(\BGL_n)$ by
\begin{align*}
    \Sq^2(\bar{c}_j) = \bar{c}_1 \bar{c}_j + (j-1)\bar{c}_{j+1}.
\end{align*}
Hence, via the commutative diagram above, we observe that the action of $\Sq^2$ on $\Ch^\bullet(\BSL_n^c)$ kills Chern classes of odd degree, and increases the indices on Chern classes of even degree. We also see that $\Sq^2(\bar{c}_n) = 0$ by consideration of degree.

It then suffices to understand how $\Sq^2$ acts on $\theta$. Via the map $\BSL_n^c \to \P^\infty$, we get a ring homomorphism
\begin{align*}
    \Z/2[t] \cong \Ch^\bullet(\P^\infty) \to \Ch^\bullet(\BSL_n^c) \cong \Z/2[\bar{c}_2, \ldots, \bar{c}_n,\theta],
\end{align*}
compatible with the Steenrod algebra action, and sending $t\mapsto \theta$. Since $\P^\infty = \BGL_1$, we understand the action of $\Sq^2$, namely it sends $\Sq^2(t) = t^2$ in $\Ch^\bullet(\P^\infty)$, and so the same happens to $\theta$ in $\Ch^\bullet(\BSL_n^c)$ as well.

Since $\Ch^1(\BSL_n^c)\cong\Z/2\Z$, there is also a twisted Steenrod square associated to the non-trivial line bundle class $\theta=[\mathcal{O}_{\mathbb{P}^\infty}(-1)]$. As usual, it differs from the untwisted Steenrod square by multiplication with $\theta$:
\[
\Sq^2_{\Theta}(x)=\theta\cdot x+\Sq^2(x).
\]
As on $\mathbb{P}^\infty$, the twisted Steenrod square maps $1\in\Ch^0$ to the class $\theta$ of the twisting line bundle, and annihilates $\theta$, cf. e.g. \cite[Section~3.6]{Wendt-Gr}. 

We summarize these results in the following.

\begin{proposition}\label{prop:action-sq2} The action of
\begin{align*}
    \Sq^2 \colon \Ch^\bullet(\BSL_n^c) &\to \Ch^{\bullet+1}(\BSL_n^c)
\end{align*}
is given by sending $\theta \mapsto \theta^2$, and
\begin{align*}
    \bar{c}_i &\mapsto \begin{cases} \bar{c}_{i+1} & 2\mid i,\ i<n \\ 0 & 2\nmid i,\ i<n \\ 0 & i=n. \end{cases}
\end{align*}
The twisted Steenrod square is given by
\[
\Sq^2_{\Theta}(x)=\theta\cdot x+\Sq^2(x).
\]
\end{proposition}

\subsection{Motivic cohomology}

As an aside, we explain a variant of the Chow-ring computation in \autoref{cor:chow-product} for motivic cohomology. 

\begin{proposition}\label{prop:motivic-cohomology} 
  The motivic cohomology of $\BSL_n^c$ is described as follows:
  \[
  H^\bullet_{\rm mot}(\BSL_n^c,\Z(\bullet))\cong H^\bullet_{\rm mot}(k,\Z(\bullet))[c_1,\dots,c_n,\theta]/(c_1-2\theta)
  \]
  The bidegrees of the generators are $|c_i|=(2i,i)$ and $|\theta|=(2,1)$.
\end{proposition}

\begin{proof}
  As in the arguments for \autoref{cor:chow-product}, we use the description of $\BSL_n^c$ as complement of the zero section of a line bundle over $\BGL_n\times\mathbb{P}^\infty$, and the associated localization sequence in motivic cohomology:
  \begin{align*}
    \cdots\to H^p(\BGL_n\times\mathbb{P}^\infty,\Z(q))\xrightarrow{c_1(f^\ast\mathcal{O}(-1))} &H^{p+2}(\BGL_n\times\mathbb{P}^\infty,\Z(q+1))\to \\
    \to&H^{p+2}(\BSL_n^c,\Z(q+1)) \xrightarrow{\partial} \cdots
  \end{align*}
  Using the projective bundle formula, we find
  \[
  H^\bullet(\BGL_n\times\mathbb{P}^\infty,\Z(\bullet))\cong H^\bullet_{\rm mot}(k,\Z(\bullet))[c_1,\dots,c_n,\theta].
  \]
  As before, the first Chern class of $f^\ast\mathcal{O}(-1)$ is $c_1(f^\ast\mathcal{O}(-1))=c_1-2\theta$, in particular, multiplication with this class is injective on motivic cohomology of $\BGL_n\times\mathbb{P}^\infty$. Consequently, the localization sequence splits up into short exact sequences showing that $H^\bullet_{\rm mot}(\BSL_n^c,\Z(\bullet))$ is the cokernel of $c_1-2\theta$ on $H^\bullet_{\rm mot}(\BGL_n,\Z(\bullet))$ as claimed.  
\end{proof}

\section{Chow--Witt groups of \texorpdfstring{$\BSL_n^c$}{BSLnc}}\label{sec:chow-witt}

In this section we compute the $\mathbf{I}^j$-cohomology of $\BSL_n^c$. Together with the results of \autoref{sec:oriented-theories}, this allows us to compute the Chow--Witt groups of $\BSL_n^c$.

\subsection{The Witt-sheaf cohomology of \texorpdfstring{$\BSL_n^c$}{BSLnc}}
\label{sec:witt-sheaf}
As a first step, we want to compute Witt-sheaf cohomology, again using the pullback square description of $\BSL_n^c$.

\begin{proposition}\label{prop:witt-cohomology-bgln} \cite[Proposition~4.5]{Wendt-Gr} The Witt-sheaf cohomology of $\BGL_n$ is given as a ${\rm W}(k)$-algebra by
\begin{align*}
    H^\bullet(\BGL_n, \mathbf{W}) &= \begin{cases}
        {\rm W}(k)[p_2, \ldots, p_{n-1}] & n\equiv 1 \pmod{2} \\
        {\rm W}(k)[p_2, \ldots, p_{n-2},e_n^2] & n\equiv 0 \pmod{2}
        \end{cases} \\
    H^\bullet(\BGL_n, \mathbf{W}(\O_{\BGL_n}(-1))) &= 
    \begin{cases}
        0 & n\equiv 1\pmod{2} \\
        H^\bullet(\BGL_n,\mathbf{W})[e_n] & n\equiv 0 \pmod{2}
    \end{cases}
\end{align*}
Here the generators are Pontryagin classes $p_{2i}$ of degree $4i$ and a potential Euler class $e_n$ in degree $n$.
Concisely we can phrase this as
\begin{align*}
    H^\bullet(\BGL_n, \mathbf{W} \oplus \mathbf{W}(-1)) \cong \begin{cases} {\rm W}(k)[p_2,p_4, \ldots, p_{n-2},e_n] & n\text{ even} \\ {\rm W}(k)[p_2,p_4, \ldots, p_{n-1}] & n\text{ odd}. \end{cases}
\end{align*}
\end{proposition}
As an example when $n=1$, the Witt-sheaf cohomology groups of ${\rm B}\mathbb{G}_{\rm m} = \P^\infty$ are given by
\begin{align*}
    H^\bullet(\P^\infty,\mathbf{W}) &= {\rm W}(k) \\
    H^\bullet(\P^\infty, \mathbf{W} (\O_{\P^\infty}(-1))) &= 0,
\end{align*}
i.e., the Witt-sheaf cohomology of $\P^\infty$ is concentrated in degree 0.

Via \autoref{ex:witt-kunneth-grassmannians} we obtain the following computation:
\begin{equation}\label{eqn:witt-cohom-product}
\begin{aligned}
    H^\bullet \left( \BGL_n \times \P^\infty, \mathbf{W} \right)&\cong H^\bullet(\BGL_n,\mathbf{W}) \\
    H^\bullet \left( \BGL_n \times \P^\infty, \mathbf{W}(\O_{\BGL_n}(-1)) \right)&\cong H^\bullet \left( \BGL_n, \mathbf{W}(\O_{\BGL_n}(-1)) \right) \\
    H^\bullet \left( \BGL_n \times \P^\infty, \mathbf{W}(\O_{\P^\infty}(-1)) \right)&\cong 0 \\
    H^\bullet \left( \BGL_n \times \P^\infty, \mathbf{W}(\O_{\P^\infty}(-1) \sqtimes \O_{\BGL_n}(-1)) \right) &\cong 0.
\end{aligned}
\end{equation}
We combine this computation and the localization sequence to compute $H^\bullet(\BSL_n^c, \mathbf{W})$, and its twisted versions. Note that for
\begin{align*}
    f := \det(-) \otimes (-)^{-2} \colon \BGL_n \times \P^\infty \to \P^\infty,
\end{align*}
we have that
\begin{align*}
    f^\ast \O_{\P^\infty}(-1) = \O_{\BGL_n}(-1) \sqtimes \O_{\P^\infty}(2).
\end{align*}
On our localization sequence, we are cupping with the Euler class of the normal bundle of the zero section of $f^\ast \O_{\P^\infty}(-1)$, which is just the bundle itself. However this bundle is of odd rank, hence its Euler class is hyperbolic, and vanishes in Witt cohomology by \cite{Levine-aspects} or the computation of Witt-sheaf cohomology in \autoref{prop:witt-cohomology-bgln} above. Hence for any line bundle $\mathcal{L} \to \BGL_n \times \P^\infty$, the localization sequence in Witt sheaf cohomology splits into short exact sequences of the form
\begin{align*}
    0 \to H^j \left( \BGL_n \times \P^\infty, \mathbf{W}(\mathcal{L}) \right) \to H^j \left( \BSL_n^c, \mathbf{W}( \left. \mathcal{L} \right|_{ \BSL_n^c }) \right) \\
    \to H^j \left( \BGL_n \times \P^\infty, \mathbf{W} (\mathcal{L} \otimes f^\ast \O(-1)) \right) \to 0.
\end{align*}
The last term is free as a $W(k)$-module by \autoref{eqn:witt-cohom-product}, hence these sequences split.

\begin{proposition}\label{prop:jth-witt-cohom-bslnc} For any line bundle $\mathcal{L} \to \BGL_n \times \P^\infty$, we obtain an isomorphism of ${\rm W}(k)$-modules, where on the left-hand side we notationally simplify $\mathcal{L}|_{\BSL_n^c}$ to $\mathcal{L}$:
\begin{align*}
  H^j \left( \BSL_n^c, \mathbf{W}\left(\mathcal{L}
  %\left. \mathcal{L} \right|_{ \BSL_n^c }
  \right) \right) \cong  H^j \left( \BGL_n \times \P^\infty, \mathbf{W}(\mathcal{L}) \right) \oplus H^j \left( \BGL_n \times \P^\infty, \mathbf{W}(\mathcal{L} \otimes f^\ast \O(-1)) \right).
\end{align*}
\end{proposition}

Now in order to compute the Witt sheaf cohomology of $\BSL_n^c$, it suffices to consider the two line bundles we care about from $\Pic(\BSL_n^c)/2$, namely the trivial one and $\O_{\P^\infty}(-1)=\Theta$.

\begin{proposition}\label{prop:wk-module-cohom-bslnc} As ${\rm W}(k)$-modules, the Witt sheaf cohomology of $\BSL_n^c$ is given by 
\begin{align*}
    H^\bullet \left( \BSL_n^c, \mathbf{W} \right) &\cong  H^\bullet\left(\BGL_n,\mathbf{W} \oplus \mathbf{W}(-1) \right) \\
    H^\bullet \left( \BSL_n^c, \mathbf{W}(\O_{\P^\infty}(-1)) \right) &\cong 0.
\end{align*}
\end{proposition}
\begin{proof} For the untwisted cohomology, using \autoref{prop:jth-witt-cohom-bslnc} we get
\begin{align*}
    H^\bullet \left( \BSL_n^c, \mathbf{W} \right)&\cong H^\bullet(\BGL_n \times \P^\infty, \mathbf{W}) \oplus H^\bullet \left( \BGL_n \times \P^\infty, \mathbf{W}(\O_{\BGL_n}(-1) \sqtimes \O_{\P^\infty}(2) \right) \\
    &\cong H^\bullet(\BGL_n \times \P^\infty, \mathbf{W}) \oplus H^\bullet \left( \BGL_n \times \P^\infty, \mathbf{W}(\O_{\BGL_n}(-1) \right) \\
    &\cong H^\bullet(\BGL_n,\mathbf{W}) \oplus H^\bullet \left( \BGL_n, \mathbf{W}(\O_{\BGL_n}(-1)) \right).
\end{align*}
For the twisted computation, again by \autoref{prop:jth-witt-cohom-bslnc} we get
\begin{align*}
    &H^\bullet \left( \BSL_n^c, \mathbf{W}(\O_{\P^\infty}(-1)) \right) \\
    &\cong  H^\bullet(\BGL_n \times \P^\infty, \mathbf{W}(\O_{\P^\infty}(-1))) \oplus H^\bullet(\BGL_n \times \P^\infty, \mathbf{W}(\O_{\BGL_n}(-1) \sqtimes \O_{\P^\infty}(-1))),
\end{align*}
which vanishes by \autoref{eqn:witt-cohom-product}.
\end{proof}

We can further describe the ring structure on $H^\bullet \left( \BSL_n^c, \mathbf{W} \right)$. Note that pullback along the map $\BSL_n^c \to \BGL_n$ induces a ring homomorphism
\begin{equation}\label{eqn:ringhom-bgln-bsln}
\begin{aligned}
    H^\bullet \left( \BGL_n, \mathbf{W} \right) \to H^\bullet(\BSL_n^c, \mathbf{W}),
\end{aligned}
\end{equation}
exhibiting $H^\bullet(\BSL_n^c,\mathbf{W})$ as an algebra over $H^\bullet \left( \BGL_n, \mathbf{W} \right)$. From this it is clear that the isomorphism in \autoref{prop:jth-witt-cohom-bslnc} is an isomorphism of $H^\bullet(\BGL_n,\mathbf{W})$-modules. By \autoref{prop:witt-cohomology-bgln}, the twisted Witt-sheaf cohomology $H^\bullet(\BGL_n,\mathbf{W}(-1))$ is a free rank 1 module over the untwisted cohomology ring $H^\bullet(\BGL_n,\mathbf{W})$, generated by the Euler class. Thus, as an $H^\bullet(\BGL_n,\mathbf{W})$-module, $H^\bullet \left( \BSL_n^c, \mathbf{W} \right)$ is a free module of rank two, generated by $1$ and the Euler class $e_n$. In order to describe the ring structure it suffices to understand what happens when the Euler class on $H^\bullet \left( \BSL_n^c, \mathbf{W} \right)$ is squared. Since \autoref{eqn:ringhom-bgln-bsln} is a ring homomorphism, this can be determined by squaring the Euler class on the cohomology of $\BGL_n$ with all twists considered. This allows us to conclude the following:

\begin{proposition}\label{prop:ring-iso-witt-cohom-bslnc} There is a ring isomorphism
\begin{align*}
    H^\bullet \left( \BSL_n^c,\mathbf{W} \right)\cong H^\bullet(\BGL_n, \mathbf{W} \oplus \mathbf{W}(-1)).
\end{align*}
\end{proposition}

We remark that this is precisely equal to the untwisted Witt sheaf cohomology of $\BSL_n$ by \cite{Wendt-Gr}. Indeed there is a natural map $\BSL_n \to \BSL_n^c$ arising from the pullback diagram in \autoref{cor:bslnc}, and it is straightforward to see that this exhibits a ring isomorphism
\begin{equation}\label{eqn:same-witt-cohom-bsln-bslnc}
\begin{aligned}
    H^\bullet(\BSL_n^c, \mathbf{W}) \xto{\sim} H^\bullet(\BSL_n,\mathbf{W}).
\end{aligned}
\end{equation}
Here we can compute $H^\bullet(\BSL_n,\mathbf{W})$ using, for example, work of Ananyevskiy \cite[Theorem~10]{Ananyevskiy-projective-bundle}.

\subsection{The \texorpdfstring{$\mathbf{I}^j$}{Ij}-cohomology of \texorpdfstring{$\BSL_n^c$}{BSLnc}}

In this section, we will describe a presentation for ${\bf I}$-cohomology of $\BSL_n^c$ as ${\rm W}(k)$-algebra. We briefly give an overview of the form our results take. Essentially, ${\bf I}$-cohomology is a direct sum of Witt-sheaf cohomology and the image of Bockstein maps
\[
\beta_{\mathcal{L}}\colon\Ch^q(\BSL_n^c)\to H^{q+1}(\BSL_n^c,{\bf I}^{q+1}(\mathcal{L})).
\]
The multiplication of Bockstein classes can be determined by reduction to $\Ch^\bullet(\BSL_n^c)$, and the formulas -- up to some subtleties involving the class $\theta$ -- largely agree with the ones for Bockstein classes on $\BGL_n$, cf. \cites{Cadek,Wendt-Gr}. Note in particular that, while we have seen in \autoref{sec:witt-sheaf} that the twisted Witt-sheaf cohomology of $\BSL_n^c$ is trivial, there are nontrivial twisted Bockstein classes. The most fundamental of these is $\beta_{\Theta}(1)$, which is the Euler class of the square-root bundle providing the quadratic orientation of the universal bundle over $\BSL_n^c$.

To establish these results, observe first that since $H^{j+1}(\BSL_n^c, \mathbf{W})$ is a free ${\rm W}(k)$-module by \autoref{prop:ring-iso-witt-cohom-bslnc} above, we get a splitting of the four-term exact sequence from~\autoref{eqn:four-term-Bar}:
\begin{align*}
    \Ch^j \left( \BSL_n^c \right) \xto{\beta_j} H^{j+1} \left( \BSL_n^c, \mathbf{I}^{j+1} \right) \to H^{j+1} \left( \BSL_n^c, \mathbf{W} \right) \to 0.
\end{align*}
Hence the ${\bf I}^j$-cohomology of $\BSL_n^c$ is given, as ${\rm W}(k)$-module, as a direct sum of its Witt-sheaf cohomology plus the image of the Bockstein homomorphism:
\begin{align}
  \label{eq:splitting}
    H^{j+1} \left( \BSL_n^c, \mathbf{I}^{j+1} \right) &\cong \im(\beta_j) \oplus H^{j+1} \left( \BSL_n^c, \mathbf{W} \right). 
\end{align}
The same statement is true for twisted ${\bf I}$-cohomology. In this case, since twisted Witt-sheaf cohomology vanishes by \autoref{prop:wk-module-cohom-bslnc}, we simply have
\begin{align*}
    H^{j+1} \left( \BSL_n^c, \mathbf{I}^{j+1}(-1) \right) &\cong \im(\beta_j)
\end{align*}
as ${\rm W}(k)$-modules. Note that, as a consequence of the B\"ar sequence, the image of the Bockstein maps is annihilated by ${\rm I}(k)\nsubgp{\rm W}(k)$, i.e., the image of Bockstein consists of 2-torsion.

To get a presentation for ${\bf I}$-cohomology, we need formulas for multiplication of Bockstein classes. The key point here is that, since the reduction homomorphism $\rho$ is injective on the image of the Bockstein by \autoref{lem:eta-squared-torsion}, it suffices to understand the image of the Steenrod square $\Sq^2$, for which we refer to \autoref{prop:action-sq2}. Products of classes in the image of $\Sq^2$ can then be multiplied using the derivation property for $\Sq^2$. For products not involving $\theta$, the formulas are the classical ones in $H^\bullet({\rm BSO}(n),\Z)$, cf. e.g. \cite{Brown}, or \cite[Proposition~7.13]{HornbostelWendt} for a motivic version. The formulas below are basically identical to the ones for $H^\bullet({\rm BO}(n),\Z^{(t)})$, cf. \cite[Lemma~4]{Cadek}, or \cite[Definition~3.15]{Wendt-Gr} for a motivic version.

\begin{proposition}
  \label{prop:multiplication-bockstein}
  The products of Bockstein classes in (total) ${\bf I}$-cohomology of $\BSL_n^c$ are given as follows:
  \begin{align}
    \label{eq:product-beta-1}
    \beta(\bar{c}_J)\cdot\beta_{\mathcal{L}}(\bar{c}_{J'})&=\sum_{k\in J}\beta(\bar{c}_{2k})\cdot P_{(J\setminus \{k\})\cap J'}\cdot \beta_{\mathcal{L}}(\bar{c}_{\Delta(J\setminus\{k\},J')})\\
    \label{eq:product-beta-2}
    \beta_{\Theta}(\bar{c}_J)\cdot\beta_{\Theta}(\bar{c}_{J'})&=\beta(\bar{c}_J)\cdot\beta(\bar{c}_{J'})+\beta_{\Theta}(1)\cdot P_{J\cap J'}\cdot\beta_{\Theta}(\bar{c}_{\Delta(J,J')})
  \end{align}
  In these formulas, $J$ and $J'$ are index sets of the form $J=\{j_1,\dots,j_l\}$ of natural numbers $0<j_1<j_2<\cdots<j_l\leq\left[\frac{1}{2}(n-1)\right]$, and $\bar{c}_J=\bar{c}_{2j_1}\cdots\bar{c}_{2j_l}$ denotes the corresponding product of even Stiefel--Whitney classes. Similarly, $P_J=\prod_{j\in J}p_{2j}$ denotes the product of the corresponding even Pontryagin classes. Finally, $\Delta(J,J')=(J\setminus J')\cup(J'\setminus J)$ on the right-hand side is the symmetric difference of index sets. The $\beta_{\mathcal{L}}$ in \autoref{eq:product-beta-1} allows to plug in the usual $\beta$ or $\beta_{\Theta}$, but of course consistently on both sides.

  The index set $J'$ can be empty, in which case $\bar{c}_\emptyset=1$, and then $\beta(1)=0$ whereas $\beta_{\Theta}(1)=e(\mathcal{O}_{\mathbb{P}^\infty}(-1))$ is the Euler class of the square-root bundle $\mathcal{O}_{\P^\infty}(-1)=\Theta$. 

  The Bockstein classes $\beta_{\mathcal{L}}(\theta\bar{c}_J)$ can be expressed as follows
  \begin{align}
    \label{eq:beta-tcj-1}
    \beta(\theta\bar{c}_J)&=\beta_{\Theta}(\bar{c}_J)\beta_{\Theta}(1)\\
    \label{eq:beta-tcj-2}
    \beta_{\Theta}(\theta\bar{c}_J)&=\beta(\bar{c}_J)\beta_{\Theta}(1).
  \end{align}
  All products of such classes can then be determined from \autoref{eq:product-beta-1} and \autoref{eq:product-beta-2} above.
\end{proposition}

\begin{proof}
  By \autoref{lem:eta-squared-torsion} and the torsion-freeness of Witt-sheaf cohomology observed above, cf.~\eqref{eq:splitting}, it suffices to check the equalities after applying the reduction map
  \[
  \rho\colon H^q(\BSL_n^c,{\bf I}^q(\mathcal{L}))\to \Ch^q(\BSL_n^c).
  \]
  We therefore only need to verify the equalities in $\Ch^\bullet(\BSL_n^c)$, with $\beta_{\mathcal{L}}$ replaced by $\Sq^2_{\mathcal{L}}$, and with the Pontryagin classes replaced by their reductions $\rho(p_{2i})=\bar{c}_{2i}^2$, cf. \cite[Theorem~6.10]{HornbostelWendt}. 

  For \autoref{eq:product-beta-1}, we note that the case $\beta_{\mathcal{L}}=\beta$ follows from the corresponding formula for ${\bf I}$-cohomology of $\BGL_n$, cf. \cite[Definition~3.15]{Wendt-Gr}. For the case with $\beta_{\Theta}$, as well as \autoref{eq:product-beta-2}, we can just replicate the proof of \cite[Lemma~4]{Cadek}, with the appropriate replacements, as follows:
  \begin{align*}
    \rho\left(\beta(\bar{c}_J)\beta_{\Theta}(\bar{c}_{J'})\right)&=\Sq^2(\bar{c}_J)\left(\Sq^2(\bar{c}_{J'})+\theta\bar{c}_{J'}\right)=\Sq^2(\bar{c}_J)\Sq^2(\bar{c}_{J'})+\Sq^2(\bar{c}_J)\cdot \theta\cdot\bar{c}_{J'}\\
    &=\sum_{k\in J}\Sq^2(\bar{c}_{2k})\Sq^2(\bar{c}_{\Delta(J\setminus\{k\},J')})\rho(P_{(J\setminus\{k\})\cap J'})\\
    &\phantom{=}+\sum_{k\in J}\Sq^2(\bar{c}_{2k})\cdot\bar{c}_{\Delta(J\setminus\{k\},J')}\cdot\rho(P_{(J\setminus\{k\})\cap J'})\cdot \theta\\
    &=\sum_{k\in J}\Sq^2(\bar{c}_{2k})\Sq^2_{\Theta}(\bar{c}_{\Delta(J\setminus\{k\},J')})\rho(P_{(J\setminus\{k\})\cap J'})
  \end{align*}
  Here we only used the definition of $\Sq^2_{\Theta}$, the derivation property for $\Sq^2$, and the fact that $\rho(p_{2i})=\bar{c}_{2i}^2$. The end result is the reduction of the right-hand side of \autoref{eq:product-beta-1}. Similarly, for \autoref{eq:product-beta-2}, we have
  \begin{align*}
    \rho\left(\beta_{\Theta}(\bar{c}_J)\beta_{\Theta}(\bar{c}_{J'})\right)&=\left(\Sq^2(\bar{c}_J)+\theta\cdot\bar{c}_J\right)\left(\Sq^2(\bar{c}_{J'})+\theta\cdot\bar{c}_{J'}\right)\\
    &=\Sq^2(\bar{c}_J)\Sq^2(\bar{c}_{J'})+\theta\left(\bar{c}_J\cdot\Sq^2(\bar{c}_{J'})+\bar{c}_{J'}\cdot\Sq^2(\bar{c}_J)\right)+\theta^2\cdot\bar{c}_J\cdot\bar{c}_{J'}\\
    &=\Sq^2(\bar{c}_J)\Sq^2(\bar{c}_{J'})+\theta\cdot\Sq^2(\bar{c}_J\bar{c}_{J'})+\theta^2\cdot\bar{c}_{\Delta(J,J')}\cdot\rho\left(P_{J\cap J'}\right)\\
    &=\Sq^2(\bar{c}_J)\Sq^2(\bar{c}_{J'})+\theta\cdot\Sq^2(\bar{c}_{\Delta(J,J')})\rho\left(P_{J\cap J'}\right)+\theta^2\cdot\bar{c}_{\Delta(J,J')}\cdot\rho\left(P_{J\cap J'}\right)\\
    &=\Sq^2(\bar{c}_J)\Sq^2(\bar{c}_{J'})+\theta\cdot\left(\Sq^2(\bar{c}_{\Delta(J,J')})+\theta\cdot\bar{c}_{\Delta(J,J')}\right)\cdot\rho\left(P_{J\cap J'}\right)\\
    &=\Sq^2(\bar{c}_J)\Sq^2(\bar{c}_{J'})+\left(\Sq^2_{\Theta}(1)\Sq^2_{\Theta}(\bar{c}_{\Delta(J,J')})\right)\cdot\rho\left(P_{J\cap J'}\right)
  \end{align*}
  Again, we have used only the standard properties for $\Sq^2$, and end up with the reduction of the right-hand side of \autoref{eq:product-beta-2}.

To show \autoref{eq:beta-tcj-1} and \autoref{eq:beta-tcj-2}, we apply the reduction technique and check the corresponding formula for Steenrod squares:
\begin{align*}
  \Sq^2(\theta\bar{c}_J)&=\theta\Sq^2(\bar{c}_J)+\Sq^2(\theta)\bar{c}_J=\left(\Sq^2(\bar{c}_J)+\theta\bar{c}_J\right)\theta=\Sq^2_{\Theta}(\bar{c}_J)\Sq^2_{\Theta}(1)\\
  \Sq^2_{\Theta}(\theta\bar{c}_J)&=\theta\Sq^2(\bar{c}_J)+\theta^2\bar{c}_J+\theta^2\bar{c}_J=\Sq^2(\bar{c}_J)\Sq^2_{\Theta}(1).\qedhere
\end{align*}
\end{proof}

\begin{remark}
  The formula in \autoref{eq:beta-tcj-1} should be compared to a similar formula for $\BGL_n$, cf. \cite[Remark~3.18]{Wendt-Gr}. In that case, we have
  \[
  \beta(\bar{c}_1\bar{c}_J)=\beta_{\det}(\bar{c}_J)\beta_{\det}(1),
  \]
  allowing to remove $\bar{c}_1$ from Stiefel--Whitney monomials. Similarly, in the case $\BSL_n^c$, \autoref{eq:beta-tcj-1} allows to express Bocksteins $\beta_{\mathcal{L}}(\theta\bar{c}_J)$ in terms of Bocksteins $\beta_{\mathcal{L}}(\bar{c}_J)$. 
  
  A similar formula for twisted Bocksteins will be important below:
  \begin{align}
    \label{eq:odd-chern-twisted}
  \beta_{\Theta}(c_{2i+1})=\beta_{\Theta}(\theta c_{2i})=\beta_{\Theta}(1)\beta(c_{2i})
  \end{align}
  Compare this to a similar formula for untwisted Steenrod squares of odd Stiefel--Whitney classes, cf. \cite[Example~3.30]{Wendt-Gr}. This formula is helpful for reducing the necessary generators for the torsion in ${\bf I}$-cohomology.\footnote{Implicitly, we already use this, since the formula in \autoref{prop:multiplication-bockstein} only concerns Bockstein classes of products of even Chern classes.}
\end{remark}

\begin{proposition}
  \label{prop:image-bockstein}
  For either of the line bundles $\mathcal{L}=\mathcal{O},\Theta$, the image of the Bockstein homomorphisms
  \begin{align*}
    \beta_{\mathcal{L}} \colon \Ch^\bullet(\BSL_n^c) \to H^\bullet(\BSL_n^c, \mathbf{I}^\bullet(\mathcal{L}))
  \end{align*}
  agrees with the ${\rm W}(k)$-torsion in ${\bf I}$-cohomology. As a module over the non-torsion part $H^\bullet(\BSL_n^c,{\bf W})$, it is generated by the Bockstein classes $\beta_{\mathcal{L}}\left(\bar{c}_J\right)$ for $\beta_{\mathcal{L}}=\beta,\beta_{\Theta}$ and $J$ running through the admissible index sets $J=\{0<j_1<\cdots <j_l\leq\left[\frac{1}{2}(n-1)\right]\}$ with $\bar{c}_J=\bar{c}_{2j_1}\cdots\bar{c}_{2j_l}$. 
\end{proposition}

\begin{proof}
  The identification of torsion as image of Bockstein follows from the splittings discussed at the start of the section, cf.~\eqref{eq:splitting}. The property that the image of the Bockstein is annihilated by the fundamental ideal is a consequence of the B\"ar sequence.

  What we need to show is that the image of Bockstein can be generated by the classes $\beta_{\mathcal{L}}(\bar{c}_J)$. To show that, we can use the same reduction technique as in the proof of \autoref{prop:multiplication-bockstein}. Since the Bockstein maps are linear, it suffices to show that all the classes $\beta_{\mathcal{L}}(m)$ for arbitrary monomials $m$ in $\theta,c_1,\dots,c_n$ are accounted for. For that, it suffices to show that any $\Sq^2_{\mathcal{L}}(m)$ is the reduction of a polynomial in $\beta_{\mathcal{L}}(\bar{c}_J)$ and  Pontryagin classes; the injectivity of $\rho$ on the image of $\beta_{\mathcal{L}}$ then shows that the original class $\beta_{\mathcal{L}}(m)$ can be rewritten to a product of generators as claimed (and possibly some Pontryagin classes from the non-torsion part).

  We deal with untwisted Steenrod squares $\Sq^2(m)$ first. We can use the derivation property to pull out squares of $\theta$ and even Chern classes $c_{2i}$, as well as odd Chern classes $c_{2i+1}$ because $\Sq^2(c_{2i+1})=0$. The class $\theta^2$ lifts to $\beta(\theta)$ and the classes $c_{2i}^2$ lift to Pontryagin classes $p_{2i}$. Since $\Sq^2(c_{2i})=c_{2i+1}$, we can also lift the odd Chern classes. Finally, we can use \autoref{eq:beta-tcj-1} to get rid of a possible remaining $\theta$ in the monomial $m$, and we're left with a monomial $\bar{c}_J$.

  Now we deal with the twisted Steenrod squares $\Sq^2_{\Theta}(m)$. We can pull out squares of even Chern classes because of
  \[
  \Sq^2_{\Theta}(c_{2i}^2x)=c_{2i}^2\Sq^2(x)+\theta c_{2i}^2x=c_{2i}^2\Sq^2_{\Theta}(x),
  \]
  where again $c_{2i}^2$ is the reduction of the Pontryagin class $p_{2i}$. Similarly, we can pull out odd Chern classes because
  \[
  \Sq^2_{\Theta}(c_{2i+1}x)=c_{2i+1}\Sq^2(x)+\theta c_{2i+1}x=c_{2i+1}\Sq^2_{\Theta}(x),
  \]
  and $c_{2i+1}$ is the reduction of $\beta(c_{2i})$.   As a special case for $x=1$, we note the resulting formula $\beta_{\Theta}(c_{2i+1})=\beta(c_{2i})\beta_{\Theta}(1)$ which already appeared in \autoref{eq:odd-chern-twisted}. Similarly, we can pull out squares of $\theta$, and get rid of any possibly remaining $\theta$ using \autoref{eq:beta-tcj-2}.
\end{proof}

We now formulate a presentation of the ${\bf I}$-cohomology ring of $\BSL_n^c$ similar to \cite[Theorem~1.1, (3)]{Wendt-Gr}. 

\begin{theorem}
  \label{thm:I-cohom-bslnc}
  The (total) ${\bf I}^\bullet$-cohomology ring
  \[
  \bigoplus_q H^q(\BSL_n^c,{\bf I}^q)\oplus H^q(\BSL_n^c,{\bf I}^q(\Theta))
  \]
  of $\BSL_n^c$ has the following presentation, as a $\Z\oplus\Z/2\Z$-graded commutative ${\rm W}(k)$-algebra:

  \begin{itemize}
  \item The cohomology ring is generated by even Pontryagin classes $p_{2i}$ in degree $(4i,0)$, for $1\leq i\leq\left[\frac{1}{2}(n-1)\right]$, the Euler class in degree $(n,0)$ and the (twisted) Bocksteins of products of Stiefel--Whitney classes
    \[
    \beta(\bar{c}_J)=\beta(\bar{c}_{2j_1}\cdots\bar{c}_{2j_l}), \qquad \beta_{\Theta}(\bar{c}_J)=\beta_{\Theta}(\bar{c}_{2j_1}\cdots\bar{c}_{2j_l})
    \]
    with the index set $J$ running through the (possibly empty) sets $\{j_1,\dots,j_r\}$ of positive natural numbers with $0<j_1<\cdots<j_l\leq\left[\frac{1}{2}(n-1)\right]$. For an index set $J=\{j_1,\dots,j_l\}$, the degree of $\beta(\bar{c}_J)$ is $\left(1+2\sum_{i=1}^lj_i,0\right)$ and the degree of $\beta_{\Theta}(\bar{c}_J)$ is $\left(1+2\sum_{i=1}^lj_i,1\right)$. 
    \item The relations satisfied in the ${\bf I}$-cohomology ring are the following, using the notation from \autoref{prop:multiplication-bockstein}:
      \begin{enumerate}[(R1)]
      \item ${\rm I}(k)\beta(\bar{c}_J)={\rm I}(k)\beta_{\Theta}(\bar{c}_J)=0$, and $\beta(\emptyset)=\beta(1)=0$.
      \item If $n=2k+1$ is odd, we have $e_{2k+1}=\beta(\bar{c}_{2k})$.
      \item For two index sets $J,J'$, where $J'$ can be empty, we have
        \begin{align*}
          \beta(\bar{c}_J)\cdot\beta_{\mathcal{L}}(\bar{c}_{J'})&=\sum_{k\in J}\beta(\bar{c}_{2k})\cdot P_{(J\setminus \{k\})\cap J'}\cdot \beta_{\mathcal{L}}(\bar{c}_{\Delta(J\setminus\{k\},J')})\\
          \beta_{\Theta}(\bar{c}_J)\cdot\beta_{\Theta}(\bar{c}_{J'})&=\beta(\bar{c}_J)\cdot\beta(\bar{c}_{J'})+\beta_{\Theta}(1)\cdot P_{J\cap J'}\cdot\beta_{\Theta}(\bar{c}_{\Delta(J,J')}).
        \end{align*}
      \end{enumerate}
  \end{itemize}
\end{theorem}

\begin{proof}
  We first note that, as discussed at the start of the section, cf.~\eqref{eq:splitting}, we have an additive splitting for either of the two line bundles $\mathcal{L}=\mathcal{O},\Theta$:
  \[
  H^q(\BSL_n^c,{\bf I}^q(\mathcal{L}))\cong\im(\beta_{\mathcal{L}}) \oplus H^{j+1} \left( \BSL_n^c, \mathbf{W}(\mathcal{L}) \right). 
  \]

We first prove that we have described all the necessary generators, i.e., our given generators actually generate the cohomology ring. By \autoref{prop:ring-iso-witt-cohom-bslnc} and \autoref{prop:witt-cohomology-bgln}, we know that the Pontryagin classes and Euler class generate the non-torsion part given by Witt-sheaf cohomology. On the other hand, \autoref{prop:image-bockstein} shows that the image of Bockstein is generated by classes $\beta(\bar{c}_J)$ and $\beta_{\Theta}(\bar{c}_J)$, possibly involving products with Pontryagin classes. Therefore, the classes we list in the presentation generate the cohomology ring. 

For the relations, we first note that all the formulas we list in (R1-3) actually hold in the cohomology ring: the Bockstein classes are torsion by \autoref{prop:image-bockstein}, the odd-rank Euler class is a Bockstein class because this is already the case for $\BGL_n$, cf. \cite{Wendt-Gr}, and the multiplication formulas in (R3) are established in \autoref{prop:multiplication-bockstein}. 

It remains to show that all relations in the cohomology ring are accounted for in our presentation. By \autoref{prop:ring-iso-witt-cohom-bslnc} and \autoref{prop:witt-cohomology-bgln}, there are no relations for non-torsion classes in Witt-sheaf cohomology.

To show that all relations between torsion classes follow from our presentation, we again use that the reduction map $\rho$ is injective on the image of $\beta_{\mathcal{L}}$. Essentially, the idea is to use the relations for Bockstein classes to reduce to a nice generating set of monomials and then check via reduction that these are linearly independent in $\Ch^\bullet(\BSL_n^c)$ by simple arguments comparing exponents appearing in monomials. The first step in this program, as in the classical arguments of Brown \cite{Brown} and \v Cadek \cite[p.~285]{Cadek}\footnote{It is interesting to note that \v Cadek in the proof in loc.~cit. redefines the odd Stiefel--Whitney classes such that the Steenrod square acts exactly as in the cohomology of $\BSL_n^c$.}, is to use the multiplication formulas (R3) to show that the (twisted) torsion part of the $\Z\oplus\Z/2\Z$-graded ${\rm W}(k)$-algebra given by the presentation in the statement will be generated by monomials
\begin{align}
  \label{eq:monomial-first-type}
  \beta_{\Theta}(1)^{2l}\prod_{i=1}^{(n-1)/2}p_{2i}^{m_i}\prod_{i=1}^{(n-1)/2}\beta(\bar{c}_{2i})^{k_i}\beta_{\Theta}(\bar{c}_J) \textrm{ with } J\neq \emptyset, \textrm{ and}\\
  \label{eq:monomial-second-type}
  \beta_{\Theta}(1)^{2l+1}\prod_{i=1}^{(n-1)/2}p_{2i}^{m_i}\prod_{i=1}^{(n-1)/2}\beta(\bar{c}_{2i})^{k_i}.
\end{align}
Applying the reduction map to $\Ch^\bullet(\BSL_n^c)$, the monomials in \eqref{eq:monomial-second-type} map to monomials containing an odd power of $\theta$ and an even power of $c_{2i}$ (from the Pontryagin classes). The elements in \eqref{eq:monomial-first-type} reduce to a sum of monomials exactly one of which contains an odd power of $\theta$:
\[
\theta^{2l}\prod_{i=1}^{(n-1)/2}\bar{c}_{2i}^{m_i}\prod_{i=1}^{(n-1)/2}\bar{c}_{2i+1}^{k_i}\left(\Sq^2(\bar{c}_J)+\theta\bar{c}_J\right),
\]
and which is uniquely determined by the numbers $l,m_i,k_i$, and contains odd powers of $\overline{c}_{2j}$ for $j\in J$. In particular, all the monomials generating the twisted torsion will have linearly independent reductions in $\Ch^\bullet(\BSL_n^c)$, showing that we captured all relations for the torsion part.
\end{proof}

\begin{remark}
  \label{rem:comparison-bgln}
  It is interesting to observe that the torsion part of the presentation for $\BSL_n^c$ is the same as for $\BGL_n^c$. However, the natural map $\BSL_n^c\to\BGL_n$ doesn't induce an isomorphism, since it annihilates $c_1$ and doesn't hit nontrivial twisted elements for $\BSL_n^c$. Moreover, there is also a difference in the Bockstein classes of odd Chern classes which is not visible from the presentation. Bocksteins of odd Chern classes are expressible in terms of the other generators, but the expressions are different in the cases $\BGL_n$ and $\BSL_n^c$. For example, for $\BGL_n$, we have $\Sq^2(c_3)=c_1c_3$ and $\Sq^2_{\det}(c_3)=0$, whereas for $\BSL_n^c$, we have $\Sq^2(c_3)=0$ and $\Sq^2_{\Theta}(c_3)=\theta c_3$.
\end{remark}

We illustrate the description of ${\bf I}^\bullet$-cohomology in a small example, and further discuss the relation between the torsion classes for $\BSL_n$, $\BSL_n^c$ and $\BGL_n$.

\begin{example}\label{ex:I-cohom-bsl4c}
  For the case $\BSL_4^c$, the ${\bf I}$-cohomology is generated by the following classes:
  \[
  p_2, e_4, \beta(\bar{c}_2), \beta_{\Theta}(1), \beta_{\Theta}(\bar{c}_2)
  \]
  The generators have (cohomological) degrees $|p_2|=4,\ |e_4|=4,\ |\beta_{\Theta}(1)| = 1,\ |\beta(\bar{c}_2)| = |\beta_{\Theta}(\bar{c}_2)| = 3$. The Witt-sheaf cohomology of $\BSL_4^c$ is the polynomial ${\rm W}(k)$-algebra generated by $p_2$ and $e_4$. The torsion part, i.e., the image of the Bockstein maps, is generated (as ${\rm W}(k)[p_2,e_4]$-module) by $\beta(\bar{c}_2)$, $\beta_{\Theta}(1)$ and $\beta_{\Theta}(\bar{c}_2)$. There aren't many relations in the torsion part, but one special case of \autoref{eq:product-beta-2} is
  \[
  \beta_{\Theta}(\bar{c}_2)^2=\beta(\bar{c}_2)^2+\beta_{\Theta}(1)^2p_2.
  \]

  We list the first few untwisted cohomology groups, explicitly with generators:
  \begin{align*}
    H^0(\BSL_4^c,\mathbf{I}^0) &\cong {\rm W}(k) \\
    H^1(\BSL_4^c, \mathbf{I}) &= 0  \\
    H^2(\BSL_4^c, \mathbf{I}^2) &\cong \Z/2\Z\langle\beta_{\Theta}(1)^2\rangle \\
    H^3(\BSL_4^c, \mathbf{I}^3) &\cong \Z/2\Z\langle\beta(\bar{c}_2)\rangle\\
    H^4(\BSL_4^c,\mathbf{I}^4)&\cong{\rm W}(k)\langle p_2,e_4\rangle\oplus\Z/2\Z\langle \beta_{\Theta}(1)^4,\beta(\theta\bar{c}_2)=\beta_{\Theta}(\bar{c}_2)\beta_{\Theta}(1)\rangle 
  \end{align*}
  Similarly, we can list the first few twisted cohomology groups:
\begin{align*}
    H^0(\BSL_4^c,\mathbf{I}^0(-1)) &=H^2(\BSL_4^c,\mathbf{I}^0(-1))=0\\
    H^1(\BSL_4^c,\mathbf{I}(-1))&\cong \Z/2\Z\langle\beta_{\Theta}(1)\rangle\\
    H^3(\BSL_4^c,\mathbf{I}(-1))&\cong \Z/2\Z\langle\beta_{\Theta}(1)^3,\beta_{\Theta}(\bar{c}_2)\rangle\\
H^4(\BSL_4^c,\mathbf{I}(-1))&\cong \Z/2\Z\langle\beta_{\Theta}(\theta\bar{c}_2)=\beta_{\Theta}(\bar{c}_3)=\beta(\bar{c}_2)\beta_{\Theta}(1)\rangle
\end{align*}

As we discussed already in \autoref{rem:comparison-bgln} above, the image of Bockstein is very similar to the case $\BGL_n$, only that $\theta$ takes over the role of $c_1$, see \cite[Example~3.30]{Wendt-Gr}. Still, some equalities of Bockstein classes are slightly different. For example, in $\BGL_3$ we have $\beta(c_3)=\beta(c_1c_2)$, which is not true for $\BSL_4^c$, where instead we have $\beta_{\Theta}(c_3)=\beta_{\Theta}(\theta c_2)$. The reason is that for $\BGL_3$ we have $\Sq^2(\bar{c}_2)=\bar{c}_1\bar{c}_2+cbar_3$, which is different from $\BSL_4^c$ where we have $\Sq^2_{\Theta}(\bar{c}_2)=\theta \bar{c}_2+ \bar{c}_3$.
\end{example}

\begin{remark}
  \label{rem:comparison-bsln}
  The natural map $\BSL_n\to \BSL_n^c$ doesn't induce an isomorphism in ${\bf I}$-cohomology. On the untwisted part of cohomology, it induces the reduction modulo the ideal $\langle\beta_{\Theta}(1)^2\rangle$. This is a consequence of the product relation \autoref{eq:product-beta-2} which implies
  \[
  \beta_{\Theta}(\bar{c}_J)^2\equiv \beta(\bar{c}_J)^2 \bmod \beta_{\Theta}(1)^2.
  \]
  A special case of this appeared with $\beta_{\Theta}(\bar{c}_2)^2=\beta(\bar{c}_2)^2+\beta_{\Theta}(1)^2p_2$ in \autoref{ex:I-cohom-bsl4c}. 
\end{remark}

\subsection{The kernel of \texorpdfstring{$\del$}{d}} The remaining piece of our computation of the Chow--Witt groups of $\BSL_n^c$ is to compute the kernel of the homomorphism $\del$.

\begin{lemma}\label{lem:ker-bockstein-untwisted} We have that the kernel of the Bockstein map
\begin{align*}
    \beta\colon \Ch^\bullet(\BSL_n^c) \to H^\bullet(\BSL_n^c, \mathbf{I}^\bullet)
\end{align*}
is given by the subring generated by $\theta^2$, odd Chern classes, squares of even Chern classes, the top Chern class $\bar{c}_n$, and Steenrod squares of products of even Chern classes:
\begin{align*}
    \Z/2\Z \left[ \theta^2, \bar{c}_{2i+1}, \bar{c}_{2i}^2, \bar{c}_n, \Sq^2 \left( \theta^e \bar{c}_{2j_1} \cdots \bar{c}_{2j_l} \right) \right] \quad\quad e\in \left\{ 0,1 \right\},\ l\ge 0.
\end{align*}
The kernel of the composite
\begin{align*}
    \del\colon \CH^\bullet(\BSL_n^c) \to \Ch^\bullet(\BSL_n^c) \to H^\bullet(\BSL_n^c, \mathbf{I}^\bullet)
\end{align*}
is given by the subring
\begin{align*}
    \Z\left[\theta^2, c_{2i+1}, c_{2i}^2, c_n, \Sq^2 \left( \theta^e c_{2j_1}\cdots c_{2j_l}\right), 2 c_{2j_1} \cdots c_{2j_l} \right]/(c_1-2\theta) \subseteq \frac{\Z[\theta,c_1, \ldots, c_n]}{(c_1-2\theta)}.
\end{align*}
\end{lemma}

\begin{proof}
  The exactness of the B\"ar sequence implies that the kernel of $\beta$ is the image of the reduction map $\rho$. On the other hand, as remarked before, the reduction map
  \[
  \rho\colon H^q(\BSL_n^c,{\bf I}^q)\to \Ch^q(\BSL_n^c)
  \]
  is injective on the image of the Bockstein map using Lemma~\ref{lem:eta-squared-torsion}. In particular, we can alternatively compute the kernel of Bockstein as the kernel of $\Sq^2$, using the description of the Steenrod square from \autoref{prop:action-sq2}. 

  Viewing $\ker\beta$ as the image of $\rho$, combined with the fact that $\rho$ is compatible with intersection products, implies immediately that $\ker\beta$ is a subring. Moreover, the generators of ${\bf I}$-cohomology as described in \autoref{thm:I-cohom-bslnc} will provide generators for $\ker \beta$. We will discuss below how the generators claimed in the lemma arise as reductions, resp. how to see they are in the kernel of $\Sq^2$. 
 
  We first observe $c_1=2\theta$ is killed in the reduction mod two map, and we recall from \autoref{prop:action-sq2} that all other odd-index Chern classes are killed by $\Sq^2$ because $c_{2i+1}=\Sq^2(c_{2i})$, thus $\Z[c_1, c_3, \ldots] \subseteq \ker(\del)$. 

  Since $\Sq^2$ is a derivation, we observe that $\Sq^2(a^2) = 2a\Sq^2(a) \equiv0 \pmod{2}$, for any cohomology class $a$. Therefore all the squares of the remaining characteristic classes lie in the kernel. Then $\theta^2=\Sq^2(\theta)$, and $c_{2i}^2$ are the reductions of Pontryagin classes $p_{2i}$.

  The top Chern class is the reduction of the Euler class, and the classes $\Sq^2(\theta^e c_J)$ are by definition in the image of the reduction map. In \autoref{thm:I-cohom-bslnc}, we can also get elements in untwisted ${\bf I}$-cohomology as products of twisted classes $\beta_{\Theta}(\bar{c}_J)$. However, using the multiplication relations in \autoref{thm:I-cohom-bslnc}, the only additional elements we can get this way can be expressed using $\Sq^2(\theta c_J)=\Sq^2_{\Theta}(1)\Sq^2_{\Theta}(c_J)$. In particular, the elements listed in the statement generate $\ker\beta$ as subring, finishing the proof.
\end{proof}

\begin{lemma}\label{lem:ker-bockstein-twisted} We have that the kernel of the twisted Bockstein map
\begin{align*}
    \beta_{\Theta}\colon \Ch^\bullet(\BSL_n^c) \to H^\bullet(\BSL_n^c, \mathbf{I}^\bullet(-1))
\end{align*}
is the sub-$(\ker\beta)$-module of $\Ch^\bullet(\BSL_n^c)$ generated by $\Sq^2_{\Theta}(\bar{c}_J)$.
\end{lemma}

\begin{proof}
  Again, we can use the B\"ar sequence to compute $\ker\beta_{\Theta}$ as image of the reduction morphism
  \[
  \rho_\Theta\colon H^q(\BSL_n^c,{\bf I}^q(-1))\to \Ch^q(\BSL_n^c).
  \]
  In the twisted case, there are no non-torsion classes, so $\ker\beta_{\Theta}$ actually agrees with the image of $\rho_{\Theta}$. The images of the generators $\beta_{\Theta}(\bar{c}_J)$ are $\Sq^2_{\Theta}(\bar{c}_J)$, and since the twisted ${\bf I}$-cohomology is generated by these as module over the untwisted ${\bf I}$-cohomology, we see that the image of $\rho_{\Theta}$ is generated by these classes as module over $\ker\beta$.
\end{proof}

\begin{remark} 
We take this opportunity to correct a small typo in the characterization of the kernel of the Bockstein homomorphism for the classifying space of the special linear group as in \cite[Theorem~6.10]{HornbostelWendt}. In the stated result, the Steenrod squares of products of even Chern classes should be in the kernel of the Bockstein, since they are in the image of $\rho$ as proven in \cite[Theorem~6.9]{HornbostelWendt}. Here is what the formulation should have been:
\end{remark}

\begin{quote}
\begin{theorem} The kernel of the Bockstein map
\begin{align*}
    \beta: \Ch^\bullet(\BSL_n) \to H^\bullet(\BSL_n, \mathbf{I}^\bullet)
\end{align*}
is given by the subring generated by odd Chern classes $\bar{c}_{2i+1}$, squares of even Chern classes $\bar{c}_{2i}^2$, the top Chern class $\bar{c}_n$, \textit{and the Steenrod squares of products of even Chern classes} $\Sq^2 \left( \bar{c}_{2j_1} \cdots \bar{c}_{2j_l} \right)$
\begin{align*}
    \Z/2\Z \left[ \bar{c}_{2i+1}, \bar{c}_{2i}^2, \bar{c}_n,\Sq^2 \left( \bar{c}_{2j_1} \cdots \bar{c}_{2j_l} \right)  \right] \subseteq \Z/2\Z \left[ \bar{c}_2, \ldots, \bar{c}_n \right].
\end{align*}
\end{theorem}
\end{quote}

\subsection{Description of the Chow--Witt ring}
\label{sec:pf-main-thm}

In this section, we will now combine our previous computations into a description of the Chow--Witt rings of the classifying spaces of ${\rm SL}_n^c$. 

\subsubsection{Fiber product description of Chow--Witt-groups}
In this section, we will give a proof of our main result \autoref{thm:CHW-computation}. Here is where we have to start being a bit careful about indices. For the Chow--Witt groups twisted by a line bundle $\mathcal{L}=\mathcal{O},\Theta$, we have a pullback square
\begin{equation}
  \label{eq:fiber-product}
  \begin{tikzcd}
    \CHW^j(\BSL_n^c,\mathcal{L})\rar\dar\pb & \ker(\del_{\mathcal{L},j})\dar\\
    H^j(\BSL_n^c,\mathbf{I}^j(\mathcal{L}))\rar["\rho" below] & \Ch^j(\BSL_n^c),
  \end{tikzcd}
\end{equation}
using \cite[Proposition~2.11]{HornbostelWendt} together with the 2-torsion-freeness of the Chow ring from \autoref{cor:chow-product}. The $\del_j$ on the top right is the map
\begin{align*}
    \del_{\mathcal{L},j}\colon \CH^j(\BSL_n^c) \to \Ch^j(\BSL_n^c) \xto{\beta_{\mathcal{L},j}} H^{j+1}(\BSL_n^c, \mathbf{I}^{j+1}(\mathcal{L})),
\end{align*}
while on the bottom left we have
\begin{align*}
    H^j(\BSL_n^c,\mathbf{I}^j(\mathcal{L})) &\cong \im(\beta_{\mathcal{L},j-1}) \oplus H^j(\BSL_n^c, \mathbf{W}(\mathcal{L})).
\end{align*}
So it's crucial to remember we're dealing with two different Bockstein homomorphisms here. As before, we typically denote the twisted Bocksteins by $\beta_{\Theta}$, with $\theta$ denoting the class of $\Theta$ in $\Ch^1(\BSL_n^c)=\Pic(\BSL_n^c)/2$.

The right vertical map is the mod 2 reduction map, while the reduction mod $\eta$ on the bottom is described as follows, cf.~\cite[Theorem~1.1.(4)]{Wendt-Gr}:
\begin{align*}
    H^\bullet(\BSL_n^c, \mathbf{I}^\bullet) &\xto{\rho} \Ch^\bullet(\BSL_n^c) \\
    p_{2i} &\mapsto \bar{c}_{2i}^2 \\
    e_n &\mapsto \bar{c}_n \\
    \beta(\bar{c}_J) &\mapsto \Sq^2(\bar{c}_J) 
    %\beta(t) = t^2 &\mapsto (\beta(t), c_2(\mathcal{L} \oplus \mathcal{L})).
\end{align*}
The latter follows from Totaro's identification $\rho\beta = \Sq^2$ as recalled in \autoref{prop:totaro-sq2}. Similarly, we have $\beta_{\Theta}(\bar{c}_J) \mapsto \Sq^2_{\Theta}(\bar{c}_J)$ in the description of the (twisted) reduction map
\[
%H^j(\BSL_n^c,{\bf I}^j(\mathcal{O}_{\mathbb{P}^{\infty}}(-1)))\to \Ch^j(\BSL_n^c)
H^j(\BSL_n^c,{\bf I}^j(\Theta))\to \Ch^j(\BSL_n^c)
\]
for $\mathcal{L}=\Theta$. %\mathcal{O}_{\mathbb{P}^\infty}(-1)$.

The additive decomposition of the ${\bf I}^j$-cohomology induces an additive decomposition of the Chow--Witt groups of $\BSL_n^c$. Explicitly, for each $j$, we will have some decomposition of $\CHW^j(\BSL_n^c,\mathcal{L})$ into pieces that look like the integers or the Grothendieck--Witt ring of the base field:
\begin{align*}
    \CHW^j(\BSL_n^c,\mathcal{L}) \cong \GW(k)^{a_j} \oplus \Z^{b_j}.
\end{align*}
The classes contributing to Grothendieck--Witt are coming from the Witt-sheaf cohomology of $\BSL_n^c$, while the classes contributing a copy of the integers are coming from the image of the Bockstein $\im(\beta_{\mathcal{L},j-1})$ or lift from 2-divisible classes in $\CH^j(\BSL_n^c)$. It suffices to compute them separately to obtain the additive structure.

\subsubsection{Generators and relations}

We now list some of the generators of the Chow--Witt rings we will need. We begin with the natural non-torsion generators, given by Chow--Witt characteristic classes:
\begin{align*}
    p_{2i} &= \left( p_{2i},  c_{2i}^2 + 2\sum_{j=\max{0,4i-n}}^{2i-1} (-1)^j c_j c_{4i-j} \right) \\
    (n\text{ even) } \quad e_n &= (e_n, c_n).
\end{align*}
Here $p_{2i}$ are the even Pontryagin classes in Chow--Witt theory, and the equality describes the Chow--Witt-theoretic Pontryagin class as an element in the fiber product, whose image in ${\bf I}$-cohomology is again the appropriate Pontryagin class, and the image in Chow theory is $c_{2i}^2 + 2\sum_{j=\max{0,4i-n}}^{2i-1} (-1)^j c_j c_{4i-j}$. Similarly, the Chow--Witt-theoretic Euler class reduces to the ${\bf I}$-cohomological Euler class and the top Chern class in Chow theory.

Note that the definition of Pontryagin classes used here is the one from \cite[Definition~5.6]{HornbostelWendt}, in terms of the symplectification morphism $\BSL_n\to\BSp_{2n}$. (These classes would then have some desired behaviour, such as stabilization and agreement of the top class with the square of the Euler class.) The formulas then follow from their counterparts in the Chow--Witt ring of $\BGL_n$ resp. $\BSL_n$, cf.~\cite[Theorem~1.1]{Wendt-Gr} and \cite[Theorem~6.10]{HornbostelWendt}. 

Next, we have the torsion generators, giving rise to the $\mathbb{Z}$-summands in the Chow--Witt groups. The natural characteristic classes in ${\bf I}$-cohomology are the Bockstein classes
\[
\beta(\overline{c}_J)=\beta(\overline{c}_{2j_1},\dots,\overline{c}_{2j_l})\in H^q(\BSL_n^c,{\bf I}^q), \qquad
\beta_{\Theta}(\overline{c}_J)\in H^q(\BSL_n^c,{\bf I}^q(\Theta))
\]
for an index set $J=\{j_1,\dots,j_l\}$ which in the second case can be empty. Since we have $\beta_{\mathcal{L}}(\overline{c}_J)\mapsto \Sq^2_{\mathcal{L}}(\overline{c}_J)$, we can choose a lift of $\Sq^2_{\mathcal{L}}(\overline{c}_J)$ along the mod 2 reduction map $\CH^q(\BSL_n^c)\to \Ch^q(\BSL_n^c)$ and denote the resulting class
\[
\widetilde{\beta}_{\mathcal{L}}(\overline{c}_J)=\left(\beta_{\mathcal{L}}(\overline{c}_J),\widetilde{\Sq}^2_{\mathcal{L}}(\overline{c}_J)\right).
\]
We could call these classes Bockstein classes in Chow--Witt theory, but it should be noted that these are not really unique in that the Chow-part of the class involves a choice of lift along $\Sq^2_{\mathcal{L}}$.

Finally, we have some classes in Chow--Witt theory which are related to the fact that the above lifts \textit{are not unique}, since a lot of classes are killed in the reduction to $\Ch^\bullet$: any class $2x\in\CH^\bullet(\BSL_n^c)$ will have trivial reduction mod 2, and consequently $(0,2x)$ will be a valid element in $\CHW^\bullet(\BSL_n^c,\mathcal{L})$, for both $\mathcal{L}=\mathcal{O},\Theta$. 

The natural way to interpret these classes is via the injection
\[
H_{\mathcal{L}}\colon \CH^q(\BSL_n^c)\to \CHW^q(\BSL_n^c,\mathcal{L})
\]
induced from the natural morphism $2{\bf K}^{\rm M}_q\to{\bf K}^{\rm MW}_q(\mathcal{L})$ of coefficient sheaves of \autoref{eqn:ses-of-sheaves}. The image of the injection agrees with the kernel of the projection $\CHW^q(\BSL_n^c,\mathcal{L})\to H^q(\BSL_n^c,{\bf I}^q(\mathcal{L}))$. The classes in the image could be called \emph{hyperbolic Chern classes}, since the injection $H_{\mathcal{L}}$ is a version of the hyperbolic morphism from algebraic to hermitian K-theory.\footnote{Also, the image of $1\in\CH^0(\BSL_n^c)\cong\Z$ under $H\colon\CH^0(\BSL_n^c)\to\widetilde{\CH}^0(\BSL_n^c)\cong\GW(k)$ is the hyperbolic plane.} The hyperbolic Chern classes are all ${\rm I}(k)$-torsion, and explain the non-uniqueness of lifts of classes along the reduction morphism $\CHW^q(\BSL_n^c,\mathcal{L})\to H^q(\BSL_n^c,{\bf I}^q(\mathcal{L}))$. In the fiber product description of Diagram~\eqref{eq:fiber-product}, the image $H_{\mathcal{L}}(x)$ an element $x\in\CH^q(\BSL_n^c)$ is identified as the tuple $(0,2x)$, i.e., the reduction to ${\bf I}$-cohomology is trivial, and the composition
\[
\CH^q(\BSL_n^c)\xrightarrow{H_{\mathcal{L}}}\CHW^q(\BSL_n^c,\mathcal{L})\to \CH^q(\BSL_n^c)
\]
is identified with multiplication by 2.

We collect the relevant information for a description of the Chow--Witt ring of $\BSL_n^c$ in the following:

\begin{theorem}
  \label{thm:chow-witt-bslnc}
  Let $k$ be a field of characteristic $\neq 2$. The Chow--Witt ring $\CHW^\bullet(\BSL_n^c,\star)$ is generated as a $\GW(k)$-algebra by the following classes:
  \begin{itemize}
  \item the even Pontryagin classes $p_{2i}\in \CHW^{4i}(\BSL_n^c,\mathcal{O})$ for $i=1,\dots,\lfloor\frac{n-1}{2}\rfloor$,
  \item the Euler class $e_n\in\CHW^n(\BSL_n^c,\mathcal{O})$ for $n$ even,
  \item the Bockstein classes
    \[
    \widetilde{\beta}_{\mathcal{L}}(\overline{c}_J)=\widetilde{\beta}_{\mathcal{L}}(\overline{c}_{2j_1}\cdots\overline{c}_{2j_l})\in \CHW^{1+\sum_{i=1}^l 2j_i}(\BSL_n^c,\mathcal{L})
    \]
    for index sets $J=\{0<j_1<\cdots<j_l\leq\lfloor\frac{n-1}{2}\rfloor\}$ which in case $\widetilde{\beta}_\Theta$ can be empty, and
    \item the hyperbolic Chern classes $H_{\mathcal{L}}(x)\in \CHW^q(\BSL_n^c,\mathcal{L})$ for $x\in\CH^q(\BSL_n^c)$.
  \end{itemize}
  The Bockstein classes and hyperbolic Chern classes are ${\rm I}(k)$-torsion. The products can be determined using the fiber product description in Diagram~\eqref{eq:fiber-product}. In particular,
  \begin{itemize}
  \item multiplication by a hyperbolic Chern class $H_{\mathcal{L}}(x)$ can be determined using the reduction morphism $\phi\colon\CHW^q(\BSL_n^c,\mathcal{L})\to\CH^q(\BSL_n^c)$ via
    \[
    H_{\mathcal{L}}(x)\cdot y=H_{\mathcal{L}}(x\cdot\phi(y)),
    \]
  \item products of Bockstein classes are determined by relation (R3) in \autoref{thm:I-cohom-bslnc}. 
  \end{itemize}
\end{theorem}

\begin{proof}
  As noted at the beginning of \autoref{sec:pf-main-thm}, the Chow--Witt groups of $\BSL_n^c$ have a fiber product description as in Diagram~\eqref{eq:fiber-product}. In particular, additive and multiplicative structure of the Chow--Witt rings can be determined from those of ${\bf I}$-cohomology and Chow theory.

  To prove the claim about generators, it suffices to show that all compatible pairs of elements from $H^q(\BSL_n^c,{\bf I}^q(\mathcal{L}))\times\ker(\partial_{\mathcal{L},q})$ are accounted for. By \autoref{thm:I-cohom-bslnc}, the ${\bf I}$-cohomology groups are generated (as ${\rm W}(k)$-algebra) by Pontryagin classes, Euler classes and Bockstein classes (in ${\bf I}$-cohomology). By the exact sequence
  \[
  \CH^q(\BSL_n^c)\xrightarrow{H_{\mathcal{L}}}\CHW^q(\BSL_n^c,\mathcal{L})\to H^q(\BSL_n^c,{\bf I}^q(\mathcal{L}))\to 0
  \]
  induced from the short exact sequence $0\to 2{\bf K}^{\rm M}_q\to{\bf K}^{\rm MW}_q(\mathcal{L})\to{\bf I}^q(\mathcal{L})\to 0$, the non-uniqueness of lifts from ${\bf I}$-cohomology to Chow--Witt groups is accounted for by the hyperbolic Chern classes. This shows that the classes listed indeed generate the Chow--Witt ring.

  The claims on multiplications follow directly from the fiber product description. For the hyperbolic Chern classes, $H_{\mathcal{L}}(x)=(0,2x)$ in the fiber product description, and consequently $H_{\mathcal{L}}(x)\cdot y=(0,2x\cdot\phi(y))=H_{\mathcal{L}}(x\cdot\phi(y))$. Any product with a Bockstein class in ${\bf I}$-cohomology can be determined by reduction to the mod 2 Chow ring by \autoref{thm:I-cohom-bslnc}, and from this we can compute any products of classes in the Chow--Witt ring. 
\end{proof}

\begin{remark}
  The statement of \autoref{thm:chow-witt-bslnc} doesn't provide a complete generators-and-relations presentation of the Chow--Witt ring. Although we have a list of generators, this is not minimal in any sense, due to the problems with hyperbolic Chern classes. Since
  \[
  H_{\mathcal{O}}\colon \CH^\bullet(\BSL_n^c)\to \CHW^\bullet(\BSL_n^c,\mathcal{O})
  \]
  fails to be a ring homomorphism, it is not enough to simply include hyperbolic Chern classes $H_{\mathcal{O}}(c_i)$. For example $H_{\mathcal{O}}(c_2)H_{\mathcal{O}}(c_4)=H_{\mathcal{O}}(2c_2c_4)$, which means that $H_{\mathcal{O}}(c_2c_4)$ has to be included among the generators and cannot be expressed as a product of other hyperbolic Chern classes. Also, it seems that a complete list of relations for such products would not provide much additional insight. For this reason, we do not make atttempts at producing a nice presentation. In any case, we want to point out that any products one wishes to compute can be evaluated using the fiber product description, combined with computations in ${\bf I}$-cohomology (where we have a complete presentation in \autoref{thm:I-cohom-bslnc}) and Chow theory (where we have a complete presentation in \autoref{cor:chow-product}). 
\end{remark}

\begin{remark}
  We make a brief remark on the key new characteristic class for $\BSL_n^c$: denoting, as before, by $\Theta$ the square-root of the determinant bundle, we have the Euler class, which in the fiber product description can be identified as $e(\Theta)=(\beta_{\Theta}(1),\theta)$. This class is not in the image of restriction from $\BGL_n$, and vanishes upon restriction to $\BSL_n$. In a way, via this Euler class, the characteristic classes of $\BSL_n^c$ have some information about the quadratic orientation of a vector bundle; the Euler class $e(\Theta)$ is an obstruction for a quadratically oriented bundle to be actually oriented, entirely related to the line bundle $\Theta$ used in the quadratic orientation.

  It is, however, important to note that for a variety $X$ with a non-trivial 2-torsion class $[\Theta]\in\Pic(X)$, we can have a non-trivial quadratic orientation $\Theta^{\otimes 2}\cong\mathcal{O}$ of the trivial line bundle $\mathcal{O}$ (or any oriented bundle, for that matter). In this case, the trivial bundle would still have a nontrivial characteristic class \emph{as a quadratically oriented bundle}. 
\end{remark}

\section{Real realizations of \texorpdfstring{$\BSL_n^c$}{BSLnc} and \texorpdfstring{$\MSL^c$}{MSLc}}\label{sec:real-realizations}

In this section we consider the image of $\BSL_n^c$ and the associated Thom spectrum $\MSL^c$ under real Betti realization $\Re_\R$.

\subsection{The real realization of $\BSL_n^c$}

We can ask what the real realization of the space $\BSL_n^c$ looks like, and how our computations here might reflect existing intuition about quadratically oriented topological vector bundles. One obstruction we confront is that the process of taking a classifying space and real realization functors do not commute. That is, for a general group scheme $G$ defined over the reals, it is not the case that $({\rm B}_\et G)(\R)$ and ${\rm B}(G(\R))$ are equivalent.

We take this opportunity to discuss the real realization of classifying spaces as described in \cite{MMW}, before returning to the case of $G=\SL_n^c$.

\begin{proposition}\label{prop:labelname} For $G$ a smooth real group scheme, we have equivalences
\begin{align*}
    {\rm Re}_{\mathbb{R}}({\rm B}_{\et}G) \simeq {\rm B}(G(\C))^{{\rm hC}_2}\simeq \bigsqcup_{[\tau]\in{\rm H}^1({\rm C}_2,G)}{\rm BAut}(\tau)(\mathbb{R}).
\end{align*}
\end{proposition}

We briefly explain the notation to unpack what the statement is saying: The left-most space ${\rm Re}_{\mathbb{R}}({\rm B}_\et G)$ is the real realization of the geometric classifying space ${\rm B}_\et G$ in motivic homotopy. The space ${\rm B}(G(\C))^{\rm hC_2}$ in the middle is the space of ${\rm C}_2$-homotopy fixed points on the classifying space ${\rm B}G(\C)$ (of the complex Lie group $G(\C)$), with the complex conjugation action. For the description on the right-hand side, the index set is the Galois cohomology group ${\rm H}^1({\rm C}_2,G)$ parametrizing strong real forms as in \cite{adams:taibi}. For a real $G$-torsor $[\tau]\in{\rm H}^1({\rm C}_2,G)$ represented by an involution $\sigma$ on $G(\C)$, the group ${\rm Aut}(\tau)$ is the automorphism group of the torsor $\tau$, and ${\rm Aut}(\tau)(\R)=G(\C)^\sigma$ is the fixed group of the involution.

As a consequence, for groups $G$ which are not special in the sense of Serre, the real realization of the geometric classifying space ${\rm B}_\et G$ will in general not be connected. One example is the orthogonal groups, in which case we have the following description, cf. \cite[Section~7]{MMW}. 

\begin{example}
  The ${\rm O}_n$-torsors on the \'etale site over $\Spec(\R)$ are exactly the isomorphism classes of rank $n$ real quadratic forms. For an ${\rm O}_n$-torsor $\tau\in{\rm H}^1(\R,{\rm O}_n)$, represented by a quadratic form of signature $(p,q)$, the real points of its automorphism group are an indefinite orthogonal group ${\rm O}(p,q)$. The real realization of the geometric classifying space of the orthogonal groups can then be identified as
\begin{align*}
    ({\rm B}_\et{\rm O}_n)(\R) = \bigsqcup_{p+q=n} {\rm B}{\rm O}(p,q).
\end{align*}
\end{example}

\begin{example} Assume $G$ is a special group in the sense of Serre, i.e., the natural map ${\rm H}^1_{\rm Zar}(X,G)\to{\rm H}^1_{\et}(X,G)$ is a bijection. Then we have
\begin{align*}
    ({\rm B}_\et G)(\R) \simeq {\rm B}(G(\R)).
\end{align*}
This applies in particular to ${\rm SL}_n$, ${\rm GL}_n$ and ${\rm Sp}_{2n}$. 
\end{example}

Note that the fact that ${\rm GL}_n$ and ${\rm SL}_n$ are special groups is essentially equivalent to Hilbert's theorem 90, the symplectic case is essentially Darboux's theorem. Note also that the description of real realization of ${\rm B}_\et G$ for special groups $G$ doesn't need the machinery of \cite{MMW} and can be deduced more directly from Krishna's equivalence in \cite{krishna:cobordism}, which identifies the geometric classifying space with the simplicial bar construction. 

As we have seen that $\SL_n^c$ is special (\autoref{prop:slnc-special}), we have the following result.

\begin{corollary}\label{cor:real-realization-bslnc}
  The natural morphism ${\rm B}_{\rm Nis}{\rm SL}_n^c\to{\rm B}_{\et}{\rm SL}_n^c$ is an equivalence, and we have an equivalence
\begin{align*}
    {\rm Re}_{\mathbb{R}}(\BSL_n^c) \simeq {\rm B}(\SL_n^c(\R)).
\end{align*}
\end{corollary}

\begin{remark}
  \label{rem:slnc-realization}
  We can determine the group $\SL_n^c(\R)$ more precisely: it is the group of pairs $(A,u)$ of a matrix $A\in\GL_n(\R)$ and a unit $u\in\R^\times$ such that $\det(A)=u^2$. There is a natural homomorphism
  \[
  \SL_n^c(\R)\to \GL_n(\R)^+\times \R^\times\colon (A,u)\to (A,u),
  \]
  induced from the fiber product definition of $\SL_n^c$, where $\GL_n(\R)^+$ denotes the group of real $n\times n$-matrices with positive determinant. By e.g. Gram--Schmidt, we can identify $\GL_n(\R)^+ \simeq {\rm SO}(n,\R)$. Consequently, we get an equivalence
  \begin{align}
    \BSL_n^c(\R)\simeq{\rm BSO}(n,\R)\times\mathbb{RP}^\infty.
  \end{align}
  This, together with the proposition above, provides an a posteriori explanation why $\BSL_n$ and $\BSL_n^c$ have the same Witt cohomology in \autoref{eqn:same-witt-cohom-bsln-bslnc}, and why their ${\bf I}$-cohomology rings exhibit some difference in the torsion.
  
  As a consequence, the natural map $\BSL_n\to\BSL_n^c$ induces the universal covering map
  \begin{align}
    {\rm Re}_{\mathbb{R}}\left(\BSL_n\right)\simeq{\rm BSO}(n,\R)\to {\rm Re}_{\mathbb{R}}\left(\BSL_n^c\right)
  \end{align}
  on real realization. This becomes a weak equivalence upon inverting 2. Note that the action of $\pi_1\left({\rm Re}_{\mathbb{R}}\left(\BSL_n^c\right)\right)\cong\Z/2\Z$ on the higher homotopy groups of ${\rm Re}_{\mathbb{R}}\left(\BSL_n^c\right)$ is trivial (by virtue of the quadratic orientation).
\end{remark}

\subsection{Jacobson realization}

Recall that since $\BGL_n$ is cellular, we can use \cite[Theorem~5.7]{HWXZ} (plus stabilization from finite-dimensional Grassmannians) to obtain an isomorphism
\begin{align*}
    H^j(\BGL_n, \mathbf{I}^j(\mathcal{L})) \xto{\sim} H^j(\BO(n);\Z(\mathcal{L})).
\end{align*}
Although $\BSL_n$ is not cellular in the sense used in loc.~cit., we still obtain an isomorphism under the realization map.\footnote{Note that again the cellularity here means stratification in terms of affine spaces, not the more general notions of cellularity of Jannsen or Totaro.} 

\begin{proposition}
  The real cycle class map of Jacobson \cite{jacobson} is an isomorphism:
\begin{align*}
    H^j(\BSL_n, \mathbf{I}^j) \xto{\sim} H^j(\BSO(n);\Z)
\end{align*}
\end{proposition}

\begin{proof}
  These groups admit the same presentation by \cite[Theorem~1.3]{HornbostelWendt}, so it suffices to check that the characteristic classes which generate the ${\bf I}^j$-cohomology are mapped to the associated characteristic classes in singular cohomology, which was shown in \cite[\S6]{HWXZ}.
\end{proof}

From \autoref{rem:slnc-realization}, we find that the real realization of $\BSL_n^c$ is $\BSO(n)\times\mathbb{RP}^\infty$, and we obtain for each of the two line bundles $\mathcal{L}=\mathcal{O},\Theta$ a commutative diagram
\[ \begin{tikzcd}
    {H^j(\BSL_n, \mathbf{I}^j)}\rar["\sim" above] & {H^j(\BSO(n);\Z)}\\
    {H^j(\BSL_n^c, \mathbf{I}^j(\mathcal{L}))}\uar \rar["\sim" below] & {H^j(\BSO(n)\times\mathbb{RP}^\infty;\Z(\mathcal{L}))} \uar
\end{tikzcd} \]
where the left vertical map is the natural one induced from the inclusion $\BSL_n\to\BSL_n^c$, the right vertical map is its real realization, and the other two morphisms are instances of Jacobson's real cycle class map. We can check that the lower-horizontal morphism is also an isomorphism: using arguments as in the proof of \autoref{thm:I-cohom-bslnc}, we can see that the singular cohomology of $\BSO(n)\times\mathbb{RP}^\infty$ has the same presentation as the ${\bf I}$-cohomology of $\BSL_n^c$. Since the latter is generated by characteristic classes which under the real cycle class map are sent to their topological counterparts, we find that the real cycle class map for $\BSL_n^c$ is also an isomorphism. As we pointed out in \autoref{rem:comparison-bsln}, the difference between cohomology of $\BSL_n$ and $\BSL_n^c$ is tied to the class $\beta_{\Theta}(1)=e(\Theta)$, which is the Euler class of the square-root line bundle $\Theta$ providing the quadratic orientation $\det\cong\Theta^{\otimes 2}$. In the real realization, the class $\beta_{\Theta}(1)=e(\Theta)$ can also be expressed as the Euler class of the pullback of the tautological line bundle on $\mathbb{RP}^\infty$.

\subsection{The Thom spectrum \texorpdfstring{${\rm MSL}^c$}{MSLc}}

Following \cite[Remark~7.11]{Hoyois-KQ}, we can define $\MSL^c$ by first defining $K^{\SL^c} \colon \mathrm{Sch}^\op \to \mathcal{S}$ by the pullback diagram
\[ \begin{tikzcd}
    K^{\SL^c}\rar\dar & K\dar["\det" right]\\
    \Pic\rar["{(-)^2}" below] & \Pic,
\end{tikzcd} \]
then defining $\MSL^c$ to be the Thom spectrum given by first taking the composition
\begin{align*}
    K^{\SL^c} \times_{\underline{\Z}} \{0\} \to K \to \Pic(\SH)
\end{align*}
and then applying the motivic Thom spectrum functor of \cite[\S16]{norms}. It is clear from this construction that $\MSL^c$ is a motivic $\mathcal{E}_\infty$ ring spectrum.

\begin{remark}
We may alternatively construct $\MSL^c$ analogously to \cite[Lemma~4.6]{bachmann:hopkins}, leveraging our construction of $\BSL_n^c$ as a \textit{metalinear Grassmannian} in \autoref{cor:BSLnc-metalinear-grassmannian}. Denoting by $\gamma_n^{{\rm SL}_n^c}$ the tautological rank $n$ quadratically oriented bundle over the metalinear Grassmannian ${\rm Gr}_n^c(\infty)$, we then get equivalences ${\rm MSL}_n^c\simeq\Sigma^{\infty-2n,n}{\rm Th}\left(\gamma_n^{{\rm SL}_n^c}\right)$ and ${\rm MSL}^c\simeq\operatornamewithlimits{colim}_n{\rm MSL}_n^c$ as in \cite[Section~4.2]{bachmann:hopkins}.
\end{remark}

\begin{remark}
  By \autoref{prop:motivic-cohomology}, the characteristic classes of a quadratically oriented bundle in motivic cohomology are Chern classes of the bundle, plus the first Chern class of the square-root bundle $\Theta$ providing the quadratic orientation. More generally, as in \cite{nandy}, we can obtain the $E$-cohomology of ${\rm MSL}^c$ for a GL-orientable theory with additive formal group law:
  \[
E^{\bullet,\bullet}({\rm MSL}^c)\cong E^{\bullet,\bullet}[\![\theta,c_1,c_2,c_3,\dots]\!]/(c_1-2\theta)
\]
  As a consequence, the motivic cohomology of $\BSL_n$ and $\BSL_n^c$ are still fairly close, and the same would be true for the spectra ${\rm MSL}$ and ${\rm MSL}^c$. This might be interesting for the evaluation of the ${\rm H}_{\rm mot}\mathbb{F}_2$-based Adams spectral sequence converging to the (suitably completed) motivic homotopy of ${\rm MSL}^c$. It would also be interesting to know where ${\rm MSL}^c$ fits in the interpolation between ${\rm MSL}$ and ${\rm MGL}$ in \cite{nandy}. 
\end{remark}

\subsection{Real realization of \texorpdfstring{${\rm MSL}^c$}{metalinear cobordism}}

In this subsection, we now want to describe the real realization of the metalinear cobordism spectrum ${\rm MSL}^c$, based on the computation of ${\rm Re}_{\mathbb{R}}(\BSL_n^c)$ above. Essentially, we follow the arguments Bachmann and Hopkins used in \cite[Section~4.2]{bachmann:hopkins} to compute the real realization of ${\rm MSL}_n$.

\begin{proposition}
  \label{prop:realization-mslnc}
  The natural morphism ${\rm MSL}_n\to{\rm MSL}^c_n$ induces an equivalence on real realization after inverting 2, i.e., there is an equivalence
  \begin{align*}
    {\rm Re}_{\mathbb{R}}({\rm MSL}_n^c)[1/2]\simeq{\rm MSO}_n[1/2].
  \end{align*}
\end{proposition}

\begin{proof}
  As in the proof of \cite[Corollary~4.7]{bachmann:hopkins}, we have 
  \begin{align*}
    {\rm MSL}_n^c\simeq\Sigma^{\infty-2n,n}{\rm cofib}\left(T_n\to \BSL^c_n\right)
  \end{align*}
  with $T_n=(\mathbb{A}^n\setminus\{0\})_{{\rm hSL}_n^c}$ the complement of the zero section of the universal rank $n$ quadratically oriented bundle. Note that here $\mathbb{A}^n\setminus\{0\}$ has the obvious ${\rm SL}_n^c$-action given by the fundamental representation ${\rm SL}_n^c\hookrightarrow{\rm GL}_n\times\mathbb{G}_{\rm m}\to{\rm GL}_n$.\footnote{Note that the representation ${\rm SL}_n^c\to{\rm GL}_n$ is not faithful, essentially it forgets about the choice of line bundle $\Theta$ in the quadratic orientation. This is in line with the usage of quadratically oriented theories: it only matters that we have Thom isomorphisms for quadratically oriented bundles, never mind the choice of quadratic orientation and the line bundle $\Theta$.} Since ${\rm SL}_n^c$ is special by \autoref{prop:slnc-special}, $T_n$ is equivalent to the bar construction of the ${\rm SL}_n^c$-space $(\mathbb{A}^n\setminus\{0\})$, using Krishna's equivalence from \cite[Proposition~3.2]{krishna:cobordism}. In particular, ${\rm Re}_{\mathbb{R}}(T_n)$ is equivalent to the bar construction of the ${\rm SL}_n^c(\R)$-space $\mathbb{R}^n\setminus\{0\}$. For a more detailed discussion of real realization of homotopy orbit spaces/quotient stacks, cf. \cite[Section~4]{MMW}.

  As observed in \autoref{rem:slnc-realization}, we have $\BSL_n^c(\R)\cong \BSO(n)\times\mathbb{RP}^\infty$, and under this identification, the ${\rm SL}_n^c(\R)$-space $\mathbb{R}^n\setminus\{0\}$ is the pullback of the universal bundle $\gamma_n$ from $\BSO(n)$. In particular, the cofiber of ${\rm Re}_{\mathbb{R}}(T_n)\to \BSL_n^c(\R)$ can be identified as
  \[
    {\rm cofib}\left({\rm Re}_{\mathbb{R}}(T_n)\to {\rm Re}_{\mathbb{R}}\left(\BSL^c_n\right)\right)\simeq {\rm cofib}\left(\gamma_n\to {\rm BSO}(n)\right) \wedge\mathbb{RP}^\infty_+
  \]
  Now, after inverting 2, the right-hand side can be identified with ${\rm MSO}_n[1/2]$, cf.~\cite[Corollary~4.7]{bachmann:hopkins} for the identification of ${\rm MSO}_n$ as real realization of ${\rm MSL}_n$. This shows that we have an equivalence ${\rm Re}_{\mathbb{R}}\left({\rm MSL}_n^c\right)[1/2]\cong{\rm MSO}_n[1/2]$.

  Note also that the realization of the natural map ${\rm SL}_n\to{\rm SL}_n^c$ is the natural inclusion ${\rm SL}_n(\R)\to{\rm SL}_n^c(\R)$. In particular, this map induces equivalences ${\rm Re}_{\mathbb{R}}({\rm MSL}_n)[1/2]\to {\rm Re}_{\mathbb{R}}({\rm MSL}_n^c)[1/2]$ as claimed.
\end{proof}

\section{Comparing \texorpdfstring{$\MSL$}{MSL} and \texorpdfstring{$\MSL^c$}{MSLc}}\label{sec:MSL-MSLc}

In this section we compare the Thom spectra $\MSL$ and $\MSL^c$ and show they are identical after inverting the Hopf element $\eta$.

\subsection{\texorpdfstring{$\eta$}{eta}-inverted theories}

As another variation, we prove a version of \autoref{prop:ring-iso-witt-cohom-bslnc} for a general ${\rm MSL}^c$-orientable cohomology theory in which $\eta$ is invertible. For the following result, we use notation from \cite{Ananyevskiy-projective-bundle} for the cohomology theory, but keep the different numbering of the Pontryagin classes. The following then describes the cohomology of $\BSL_n^c$, closely resembling \cite[Theorem~10]{Ananyevskiy-projective-bundle}, the second part is also proved in \cite[Corollary~5.3.4]{haution}.

\begin{theorem}
  \label{thm:eta-inverted-bslnc}
  Assume $A^{\bullet,\bullet}$ is a representable ${\rm SL}^c$-orientable cohomology theory, and denote $A^\bullet(X)=A^{\bullet,0}_\eta(X)$. Then we have
  \begin{align*}
  A^\bullet(\BSL_{2n}^c)& \cong A^\bullet({\rm pt})[\![p_2,\dots,p_{2n-2},e_{2n}]\!]_h\\
  A^\bullet(\BSL_{2n+1}^c)& \cong A^\bullet({\rm pt})[\![p_2,\dots,p_{2n}]\!]_h
  \end{align*}
  The twisted $A$-cohomology of $\BSL_n^c$ vanishes. In fact, the universal covering map $\BSL_n\to\BSL_n^c$ induces an isomorphism on $A$-cohomology, mapping $\SL^c$-characteristic classes to $\SL$-characteristic classes.

  More generally, the same is true for representable SL-orientable cohomology theories.
\end{theorem}

\begin{proof}
  The proof uses the same arguments as the corresponding result for Witt-sheaf cohomology in \autoref{sec:witt-sheaf}. The Witt-sheaf computation for $\BGL_n$ in \cite{Wendt-Gr} can be extended to general $\eta$-inverted theories as follows: the vanishing of reduced $A$-cohomology of $\mathbb{P}^\infty$ can be seen as in \cite[Lemma~3.7]{nandy}, and building on this as base of an induction, the argument of \cite[Proposition~4.8]{Wendt-Gr} goes through. This provides the generalization of \autoref{prop:witt-cohomology-bgln}, and the rest of the arguments in \autoref{sec:witt-sheaf} goes through. Alternatively, of course, one can follow the arguments for \cite[Theorem~10]{Ananyevskiy-projective-bundle} for $\BSL_n^c$. 

The claim that $\BSL_n\to\BSL_n^c$ induces an isomorphism on $A$-cohomology is the analogue of \autoref{eqn:same-witt-cohom-bsln-bslnc}, in the discussion after \autoref{prop:ring-iso-witt-cohom-bslnc}. The polynomial generators for $A$-cohomology of $\BSL_n$ and $\BSL_n^c$ are both induced by pullback from the Pontryagin and Euler classes on $\BGL_n$ via the natural covering maps to $\BGL_n$.
  
  The extension to SL-oriented theories uses Ananyevskiy's theorem that SL-oriented theories have Thom isomorphisms for ${\rm SL}^c$-bundles, cf.~\cite[Theorem~1.1]{Ananyevskiy}. With this, all the arguments above go through. The key point there is that we have a localization sequence for the universal quadratically oriented oriented rank $n$ bundle on $\BSL_n^c$ which together with the Thom isomorphism is used in the induction arguments.
\end{proof}

\begin{corollary}
  \label{cor:a-hlgy-mslc}
  Assume $A^{\bullet,\bullet}$ is a representable ${\rm SL}^c$-orientable cohomology theory, and denote $A^\bullet(X)=A^{\bullet,0}_\eta(X)$. Then the $A$-homology of the infinite metalinear Grassmannian $\Gr^c(2n+1,\infty)$ is the topological dual of the $A$-cohomology. The natural morphism
  \begin{align*}
    \widetilde{\Gr}(2n+1,\infty)\to\Gr^c(2n+1,\infty)
  \end{align*}
  induces an isomorphism in $A$-homology. In particular,
  \[
  A_\bullet(\MSL^c)\simeq\colim_n A_\bullet(\Gr^c(2n+1,\infty))\simeq A_\bullet[e_2,e_4,\dots]
  \]
  where we use the names for polynomial generators from \cite[Theorem~4.1(2)]{bachmann:hopkins}.
\end{corollary}

\begin{proof}
  As in (the end of) the proof of \cite[Lemma~4.16]{bachmann:hopkins}, we can use the strong dualizability and cellularity from \autoref{prop:cellular-dualizable}, \cite[Corollary~4.10]{bachmann:hopkins} and \cite[Remark~14]{Ananyevskiy-projective-bundle} to see the claim on topological duals. From this and \autoref{thm:eta-inverted-bslnc}, we get that $\widetilde{\Gr}(2n+1,\infty)\to \Gr^c(2n+1,\infty)$ induces an isomorphism in $A$-homology. The claim about $A$-homology of $\MSL^c$ follows from the corresponding claim for $\MSL$ in \cite[Theorem~4.1(2)]{bachmann:hopkins}.
\end{proof}

\begin{remark}
  As a particular consequence, one can compute the cohomology operations for ${\rm MSL}^c[\eta^{-1}]$ from this result, as endomorphisms of ${\rm MSL}^c[\eta^{-1}]$ in $\mathcal{SH}[\eta^{-1}]$, using \autoref{thm:eta-inverted-bslnc} and Thom isomorphisms:
  \[
  \left({\rm MSL}^c_\eta\right)^{\bullet,\bullet}\left({\rm MSL}_\eta^c\right)\cong \left({\rm MSL}^c_\eta\right)^{\bullet,\bullet}[\![p_2,p_4,\dots]\!]
  \]
\end{remark}

We now want to use the computation of cohomology of $\BSL_n^c$ for $\eta$-inverted quadratically oriented cohomology theories to compare MSL and ${\rm MSL}^c$, using ideas from \cite{bachmann:hopkins}.

\begin{corollary}
  \label{cor:comparison-mslc}
  Let $k$ be a field of characteristic $\neq 2$. The natural morphism ${\rm MSL}\to{\rm MSL}^c$ becomes an equivalence in the $\eta$-inverted stable motivic homotopy $\mathcal{SH}(k)[\eta^{-1}]$. In particular, we also have 
  \[
  \underline{\pi}_*{\rm MSL}^c[\eta^{-1}]\cong\underline{\bf W}[y_2,y_4,\dots].
  \]
\end{corollary}

\begin{proof}
  As in the homotopy computations in \cite[Section~8]{bachmann:hopkins}, it suffices to check that the map $\nu\colon{\rm MSL}\to{\rm MSL}^c$ is an equivalence both after inverting 2 and after completing at 2. We can check $\nu[1/2]$ is an equivalence on real realizations, using \autoref{prop:realization-mslnc}.

  It remains to show that $\nu_{(2)}$ is an equivalence. Recall from \cite[Theorem~1.1]{bachmann:hopkins} that there is a resolution
  \[
  \mathbb{S}[\eta^{-1}]_{(2)}\to {\rm kw}_{(2)}\xrightarrow{\varphi}\Sigma^4{\rm kw}_{(2)}
  \]
  of the $\eta$-local sphere in terms of the connective Balmer--Witt K-theory spectrum kw. Because of this resolution, it suffices to show that the natural morphism ${\rm MSL}\to{\rm MSL}^c$ induces an equivalence in (2-completed) kw-homology, compatible with the map $\varphi$. Using \autoref{cor:a-hlgy-mslc}, we have isomorphisms
  \[
    {\rm kw}_*[e_2,e_4,\dots] \cong \pi_*({\rm kw}\wedge{\rm MSL})\xrightarrow{\cong} \pi_*({\rm kw}\wedge{\rm MSL}^c).
  \]
Since this isomorphism is in fact induced (via Thom isomorphism) from an kw-homology isomorphism ${\rm kw}_*(\widetilde{\Gr}(2n+1,\infty))\to{\rm kw}_*(\Gr^c(2n+1,\infty))$, we see that the polynomial generators arise from the cells of $\mathbb{HP}^\infty$, both for MSL and ${\rm MSL}^c$. In particular, the kw-homology isomorphism above is compatible with the map $\varphi$, which shows that $\nu_{(2)}$ is an equivalence.

The claim on the homotopy sheaves is then immediate from \cite[Theorem~8.8]{bachmann:hopkins}, but can alternatively also be proved just as in loc.~cit.
\end{proof}

\begin{remark}
  It is well-known that SL-orientations and ${\rm SL}^c$-orientations are very closely related. One instance of this is Ananyevskiy's theorem that SL-oriented theories have Thom isomorphisms for ${\rm SL}^c$-bundles, cf.~\cite[Theorem~1.1]{Ananyevskiy}, which we have already used above. The $\eta$-local equivalence ${\rm MSL}[\eta^{-1}]\to {\rm MSL}^c[\eta^{-1}]$ is another version of this close connection, showing that, after inverting $\eta$, there is really no difference between SL-orientations and ${\rm SL}^c$-orientations.
\end{remark}

\printbibliography

@article {AHW2,
    AUTHOR = {Asok, Aravind and Hoyois, Marc and Wendt, Matthias},
     TITLE = {Affine representability results in {$\mathbb{A}^1$}-homotopy
              theory, {II}: {P}rincipal bundles and homogeneous spaces},
   JOURNAL = {Geom. Topol.},
  FJOURNAL = {Geometry \& Topology},
    VOLUME = {22},
      YEAR = {2018},
    NUMBER = {2},
     PAGES = {1181--1225},
      ISSN = {1465-3060},
   MRCLASS = {14F42 (14L10 20G15 55R15)},
  MRNUMBER = {3748687},
MRREVIEWER = {Federico Binda},
       DOI = {10.2140/gt.2018.22.1181},
       URL = {https://doi-org.proxy.library.upenn.edu/10.2140/gt.2018.22.1181},
}

@article {HornbostelWendt,
    AUTHOR = {Hornbostel, Jens and Wendt, Matthias},
     TITLE = {Chow--{W}itt rings of classifying spaces for symplectic and
              special linear groups},
   JOURNAL = {J. Topol.},
  FJOURNAL = {Journal of Topology},
    VOLUME = {12},
      YEAR = {2019},
    NUMBER = {3},
     PAGES = {916--966},
      ISSN = {1753-8416,1753-8424},
   MRCLASS = {14F42 (14C17 19D45 19G12 20G10 55R40)},
  MRNUMBER = {4072161},
MRREVIEWER = {Masaki\ Kameko},
       DOI = {10.1112/topo.12103},
       URL = {https://doi.org/10.1112/topo.12103},
}

@incollection {Ananyevskiy,
    AUTHOR = {Ananyevskiy, Alexey},
     TITLE = {{$\textit{SL}$}-oriented cohomology theories},
 BOOKTITLE = {Motivic homotopy theory and refined enumerative geometry},
    SERIES = {Contemp. Math.},
    VOLUME = {745},
     PAGES = {1--19},
 PUBLISHER = {Amer. Math. Soc., [Providence], RI},
      YEAR = {2020},
      ISBN = {978-1-4704-4898-1},
   MRCLASS = {14F42 (55R40)},
  MRNUMBER = {4071210},
MRREVIEWER = {Annette\ Huber},
       DOI = {10.1090/conm/745/15020},
       URL = {https://doi.org/10.1090/conm/745/15020},
}

@article {Brown,
    AUTHOR = {Brown, Jr., Edgar H.},
     TITLE = {The cohomology of {$B{\rm SO}\sb{n}$} and {$B{\rm O}\sb{n}$}
              with integer coefficients},
   JOURNAL = {Proc. Amer. Math. Soc.},
  FJOURNAL = {Proceedings of the American Mathematical Society},
    VOLUME = {85},
      YEAR = {1982},
    NUMBER = {2},
     PAGES = {283--288},
      ISSN = {0002-9939,1088-6826},
   MRCLASS = {55R40},
  MRNUMBER = {652459},
MRREVIEWER = {R.\ E.\ Stong},
       DOI = {10.2307/2044298},
       URL = {https://doi.org/10.2307/2044298},
}

@article {Cadek,
    AUTHOR = {\v Cadek, Martin},
     TITLE = {The cohomology of {${\rm BO}(n)$} with twisted integer
              coefficients},
   JOURNAL = {J. Math. Kyoto Univ.},
  FJOURNAL = {Journal of Mathematics of Kyoto University},
    VOLUME = {39},
      YEAR = {1999},
    NUMBER = {2},
     PAGES = {277--286},
      ISSN = {0023-608X},
   MRCLASS = {55R40 (57R20)},
  MRNUMBER = {1709293},
MRREVIEWER = {Hans-Werner\ Henn},
       DOI = {10.1215/kjm/1250517912},
       URL = {https://doi.org/10.1215/kjm/1250517912},
}

@book {Hatcher,
    AUTHOR = {Hatcher, Allen},
     TITLE = {Algebraic topology},
 PUBLISHER = {Cambridge University Press, Cambridge},
      YEAR = {2002},
     PAGES = {xii+544},
   MRCLASS = {55-01 (55-00)},
  MRNUMBER = {1867354},
MRREVIEWER = {Donald\ W.\ Kahn},
}

@article{HMW,
      title={Chow--{W}itt rings and topology of flag varieties}, 
      author={Thomas Hudson and Ákos K. Matszangosz and Matthias Wendt},
      journal={J. Topol.},
      volume={17},
      number={4},
      doi={10.1112/topo.70004},
      url={https://doi.org/10.1112/topo.70004},
      pages={78pp},
      year={2024},
      eprint={2302.11003},
      archivePrefix={arXiv},
      primaryClass={math.AG}
}

@misc{MMW,
title = {Equivariant real cycle class map and Witt-sheaf cohomology of classifying spaces},
author = {Lorenzo Mantovani and \'Akos K. Matszangosz and Matthias Wendt},
year = {2025},
note = {hopefully on the arXiv soon}
}

@incollection {Totaro,
    AUTHOR = {Totaro, Burt},
     TITLE = {The {C}how ring of a classifying space},
 BOOKTITLE = {Algebraic {$K$}-theory ({S}eattle, {WA}, 1997)},
    SERIES = {Proc. Sympos. Pure Math.},
    VOLUME = {67},
     PAGES = {249--281},
 PUBLISHER = {Amer. Math. Soc., Providence, RI},
      YEAR = {1999},
      ISBN = {0-8218-0927-X},
   MRCLASS = {14C15 (14L30 55N91)},
  MRNUMBER = {1743244},
MRREVIEWER = {Dan\ Edidin},
       DOI = {10.1090/pspum/067/1743244},
       URL = {https://doi.org/10.1090/pspum/067/1743244},
}

@book {Fulton,
    AUTHOR = {Fulton, William},
     TITLE = {Intersection theory},
    SERIES = {Ergebnisse der Mathematik und ihrer Grenzgebiete. 3. Folge. A
              Series of Modern Surveys in Mathematics [Results in
              Mathematics and Related Areas. 3rd Series. A Series of Modern
              Surveys in Mathematics]},
    VOLUME = {2},
   EDITION = {Second},
 PUBLISHER = {Springer-Verlag, Berlin},
      YEAR = {1998},
     PAGES = {xiv+470},
   MRCLASS = {14C17 (14-02)},
  MRNUMBER = {1644323},
       DOI = {10.1007/978-1-4612-1700-8},
       URL = {https://doi.org/10.1007/978-1-4612-1700-8},
}

@article {OkTel,
    AUTHOR = {Okonek, Christian and Teleman, Andrei},
     TITLE = {Intrinsic signs and lower bounds in real algebraic geometry},
   JOURNAL = {J. Reine Angew. Math.},
  FJOURNAL = {Journal f\"{u}r die Reine und Angewandte Mathematik. [Crelle's
              Journal]},
    VOLUME = {688},
      YEAR = {2014},
     PAGES = {219--241},
      ISSN = {0075-4102},
   MRCLASS = {14P05},
  MRNUMBER = {3176620},
MRREVIEWER = {Thomas C. Craven},
       DOI = {10.1515/crelle-2012-0055},
       URL = {https://doi.org/10.1515/crelle-2012-0055},
}

@book {Morel,
    AUTHOR = {Morel, Fabien},
     TITLE = {{$\mathbb{A}^1$}-algebraic topology over a field},
    SERIES = {Lecture Notes in Mathematics},
    VOLUME = {2052},
 PUBLISHER = {Springer, Heidelberg},
      YEAR = {2012},
     PAGES = {x+259},
      ISBN = {978-3-642-29513-3},
   MRCLASS = {14F35 (14F05)},
  MRNUMBER = {2934577},
MRREVIEWER = {Matthias Wendt},
       DOI = {10.1007/978-3-642-29514-0},
       URL = {https://doi.org/10.1007/978-3-642-29514-0},
}

@article {equivariant-CHW,
    AUTHOR = {Di Lorenzo, Andrea and Mantovani, Lorenzo},
     TITLE = {Equivariant {C}how--{W}itt groups and moduli stacks of elliptic
              curves},
   JOURNAL = {Doc. Math.},
    VOLUME = {28},
      YEAR = {2023},
    NUMBER = {2},
     PAGES = {315--368},
       DOI = {10.4171/dm/911},
       URL = {https://doi.org/10.4171/dm/911},
      eprint={2107.02305},
      archivePrefix={arXiv},
      primaryClass={math.AG}
}

@article{Krishna,
author = {Amalendu Krishna},
title = {{Higher {C}how groups of varieties with group action}},
volume = {7},
journal = {Algebra \& Number Theory},
number = {2},
publisher = {MSP},
pages = {449 -- 506},
year = {2013},
doi = {10.2140/ant.2013.7.449},
URL = {https://doi.org/10.2140/ant.2013.7.449}
}

@article{Ananyevskiy-projective-bundle,
 title={The special linear version of the projective bundle theorem},
 volume={151},
 ISSN={0010-437X, 1570-5846},
 DOI={10.1112/S0010437X14007702},
 abstractNote={A special linear Grassmann variety SGr(k, n) is the complement to the zero section of the determinant of the tautological vector bundle over Gr(k, n). For an SL-oriented representable ring cohomology theory A∗(−) with invertible stable Hopf map η, including Witt groups and MSL∗η,∗, we have A∗(SGr(2, 2n + 1)) ∼= A∗(pt)[e]/(e2n), and A∗(SGr(k, n)) is a truncated polynomial algebra over A∗(pt) whenever k(n − k) is even. A splitting principle for such theories is established. Using the computations for the special linear Grassmann varieties, we obtain a description of A∗(BSLn) in terms of homogeneous power series in certain characteristic classes of tautological bundles.},
 number={3},
 journal={Compositio Mathematica},
 author={Ananyevskiy, Alexey},
 year={2015},
 month=mar,
 pages={461–501} }

@article {Levine-aspects,
    AUTHOR = {Levine, Marc},
     TITLE = {Aspects of enumerative geometry with quadratic forms},
   JOURNAL = {Doc. Math.},
  FJOURNAL = {Documenta Mathematica},
    VOLUME = {25},
      YEAR = {2020},
     PAGES = {2179--2239},
      ISSN = {1431-0635,1431-0643},
   MRCLASS = {14F42 (14C17)},
  MRNUMBER = {4198841},
MRREVIEWER = {Stephen\ McKean},
}

@misc{bachmann:hopkins,
title={$\eta$-periodic motivic stable homotopy theory over fields},
author={Bachmann, Tom and Hopkins, Michael J.},
year={2020},
eprint={2005.06778v4},
archivePrefix={arXiv},
primaryClass={math.KT}
}

@incollection {Jannsen,
    AUTHOR = {Jannsen, Uwe},
     TITLE = {Motivic sheaves and filtrations on {C}how groups},
 BOOKTITLE = {Motives ({S}eattle, {WA}, 1991)},
    SERIES = {Proc. Sympos. Pure Math.},
    VOLUME = {55, Part 1},
     PAGES = {245--302},
 PUBLISHER = {Amer. Math. Soc., Providence, RI},
      YEAR = {1994},
      ISBN = {0-8218-1636-5},
   MRCLASS = {14C15 (11G09 14A20 14C25)},
  MRNUMBER = {1265533},
MRREVIEWER = {S.\ L.\ Kleiman},
       DOI = {10.1090/pspum/055.1/1265533},
       URL = {https://doi.org/10.1090/pspum/055.1/1265533},
}

@article {Totaro-motive,
    AUTHOR = {Totaro, Burt},
     TITLE = {The motive of a classifying space},
   JOURNAL = {Geom. Topol.},
  FJOURNAL = {Geometry \& Topology},
    VOLUME = {20},
      YEAR = {2016},
    NUMBER = {4},
     PAGES = {2079--2133},
      ISSN = {1465-3060,1364-0380},
   MRCLASS = {14C15 (14A20 14F42 14M20 55R35)},
  MRNUMBER = {3548464},
MRREVIEWER = {Florence\ Lecomte},
       DOI = {10.2140/gt.2016.20.2079},
       URL = {https://doi.org/10.2140/gt.2016.20.2079},
}

@article {Totaro-Witt,
    AUTHOR = {Totaro, Burt},
     TITLE = {Non-injectivity of the map from the {W}itt group of a variety
              to the {W}itt group of its function field},
   JOURNAL = {J. Inst. Math. Jussieu},
  FJOURNAL = {Journal of the Institute of Mathematics of Jussieu. JIMJ.
              Journal de l'Institut de Math\'ematiques de Jussieu},
    VOLUME = {2},
      YEAR = {2003},
    NUMBER = {3},
     PAGES = {483--493},
      ISSN = {1474-7480,1475-3030},
   MRCLASS = {19G12 (19E15)},
  MRNUMBER = {1990223},
MRREVIEWER = {Charles\ Weibel},
       DOI = {10.1017/S1474748003000136},
       URL = {https://doi.org/10.1017/S1474748003000136},
}

@article {Fasel-IJ,
    AUTHOR = {Fasel, Jean},
     TITLE = {The projective bundle theorem for {${\bf I}^j$}-cohomology},
   JOURNAL = {J. K-Theory},
  FJOURNAL = {Journal of K-Theory. K-Theory and its Applications in Algebra,
              Geometry, Analysis \& Topology},
    VOLUME = {11},
      YEAR = {2013},
    NUMBER = {2},
     PAGES = {413--464},
      ISSN = {1865-2433,1865-5394},
   MRCLASS = {14F43 (19G38)},
  MRNUMBER = {3061003},
MRREVIEWER = {Feng-Wen\ An},
       DOI = {10.1017/is013002015jkt217},
       URL = {https://doi.org/10.1017/is013002015jkt217},
}

@article {Fasel-CHW,
    AUTHOR = {Fasel, Jean},
     TITLE = {Groupes de {C}how--{W}itt},
   JOURNAL = {M\'{e}m. Soc. Math. Fr. (N.S.)},
  FJOURNAL = {M\'{e}moires de la Soci\'{e}t\'{e} Math\'{e}matique de France.
              Nouvelle S\'{e}rie},
    NUMBER = {113},
      YEAR = {2008},
     PAGES = {viii+197},
      ISSN = {0249-633X,2275-3230},
      ISBN = {978-2-85629-262-4},
   MRCLASS = {14C15 (13C10 14C17 18F30)},
  MRNUMBER = {2542148},
MRREVIEWER = {Baptiste\ Calm\`es},
       DOI = {10.24033/msmf.425},
       URL = {https://doi.org/10.24033/msmf.425},
}

@article {BM,
    AUTHOR = {Barge, Jean and Morel, Fabien},
     TITLE = {Groupe de {C}how des cycles orient\'{e}s et classe d'{E}uler
              des fibr\'{e}s vectoriels},
   JOURNAL = {C. R. Acad. Sci. Paris S\'{e}r. I Math.},
  FJOURNAL = {Comptes Rendus de l'Acad\'{e}mie des Sciences. S\'{e}rie I.
              Math\'{e}matique},
    VOLUME = {330},
      YEAR = {2000},
    NUMBER = {4},
     PAGES = {287--290},
      ISSN = {0764-4442},
   MRCLASS = {14C15},
  MRNUMBER = {1753295},
       DOI = {10.1016/S0764-4442(00)00158-0},
       URL = {https://doi.org/10.1016/S0764-4442(00)00158-0},
}

@article {AF-comparing,
    AUTHOR = {Asok, A. and Fasel, J.},
     TITLE = {Comparing {E}uler classes},
   JOURNAL = {Q. J. Math.},
  FJOURNAL = {The Quarterly Journal of Mathematics},
    VOLUME = {67},
      YEAR = {2016},
    NUMBER = {4},
     PAGES = {603--635},
      ISSN = {0033-5606,1464-3847},
   MRCLASS = {14C15 (14F35)},
  MRNUMBER = {3609848},
MRREVIEWER = {Daniel\ C.\ Isaksen},
       DOI = {10.1093/qmath/haw033},
       URL = {https://doi.org/10.1093/qmath/haw033},
}

@article {EulerBW3,
    AUTHOR = {Bachmann, Tom and Wickelgren, Kirsten},
     TITLE = {Euler classes: six-functors formalism, dualities, integrality
              and linear subspaces of complete intersections},
   JOURNAL = {J. Inst. Math. Jussieu},
  FJOURNAL = {Journal of the Institute of Mathematics of Jussieu. JIMJ.
              Journal de l'Institut de Math\'{e}matiques de Jussieu},
    VOLUME = {22},
      YEAR = {2023},
    NUMBER = {2},
     PAGES = {681--746},
      ISSN = {1474-7480,1475-3030},
   MRCLASS = {14F42 (19E15 55R40)},
  MRNUMBER = {4557905},
       DOI = {10.1017/S147474802100027X},
       URL = {https://doi.org/10.1017/S147474802100027X},
}

@article {Wendt-Gr,
    AUTHOR = {Wendt, Matthias},
     TITLE = {Chow--{W}itt rings of {G}rassmannians},
   JOURNAL = {Algebr. Geom. Topol.},
  FJOURNAL = {Algebraic \& Geometric Topology},
    VOLUME = {24},
      YEAR = {2024},
    NUMBER = {1},
     PAGES = {1--48},
      ISSN = {1472-2747,1472-2739},
   MRCLASS = {14C15 (14C17 14F43 14M15)},
  MRNUMBER = {4721362},
       DOI = {10.2140/agt.2024.24.1},
       URL = {https://doi.org/10.2140/agt.2024.24.1},
}

@article {HWXZ,
    AUTHOR = {Hornbostel, Jens and Wendt, Matthias and Xie, Heng and
              Zibrowius, Marcus},
     TITLE = {The real cycle class map},
   JOURNAL = {Ann. K-Theory},
  FJOURNAL = {Annals of K-Theory},
    VOLUME = {6},
      YEAR = {2021},
    NUMBER = {2},
     PAGES = {239--317},
      ISSN = {2379-1683,2379-1691},
   MRCLASS = {14C15 (14F25 14P25 19G12)},
  MRNUMBER = {4301905},
MRREVIEWER = {Satoshi\ Mochizuki},
       DOI = {10.2140/akt.2021.6.239},
       URL = {https://doi.org/10.2140/akt.2021.6.239},
}

@article {haution,
    AUTHOR = {Haution, Olivier},
     TITLE = {Motivic {P}ontryagin classes and hyperbolic orientations},
   JOURNAL = {J. Topol.},
    VOLUME = {16},
      YEAR = {2023},
    NUMBER = {4},
     PAGES = {1423--1474},
       DOI = {10.1112/topo.12317},
       URL = {https://doi.org/10.1112/topo.12317},
}

@article {merkurjev:special,
    AUTHOR = {Merkurjev, Alexander},
     TITLE = {Classification of special reductive groups},
   JOURNAL = {Michigan Math. J.},
    VOLUME = {72},
      YEAR = {2022},
     PAGES = {529--541},
       DOI = {10.1307/mmj/20207201},
       URL = {https://doi.org/10.1307/mmj/20207201},
}

@article {adams:taibi,
    AUTHOR = {Adams, Jeffrey and Ta\"ibi, Olivier},
     TITLE = {Galois and {C}artan cohomology of real groups},
   JOURNAL = {Duke Math. J.},
    VOLUME = {167},
      YEAR = {2018},
    NUMBER = {6},
     PAGES = {1057--1097},
       DOI = {10.1215/00127094-2017-0052},
       URL = {https://doi.org/10.1215/00127094-2017-0052},
}

@misc{krishna:cobordism,
author = {Krishna, Amalendu},
title = {The motivic cobordism for group actions},
year = {2012},
eprint={1206.5952},
archivePrefix={arXiv},
primaryClass={math.AG}
}

@article {jacobson,
    AUTHOR = {Jacobson, Jeremy A.},
     TITLE = {Real cohomology and the powers of the fundamental ideal in the
              {W}itt ring},
   JOURNAL = {Ann. K-Theory},
    VOLUME = {2},
      YEAR = {2017},
    NUMBER = {3},
     PAGES = {357--385},
       DOI = {10.2140/akt.2017.2.357},
       URL = {https://doi.org/10.2140/akt.2017.2.357},
}

@article {asok:fasel:secondary,
    AUTHOR = {Asok, Aravind and Fasel, Jean},
     TITLE = {Secondary characteristic classes and the {E}uler class},
   JOURNAL = {Doc. Math.},
      YEAR = {2015},
     PAGES = {7--29}
     }

@misc{nandy,
author = {Nandy, Ahina},
title = {An interpolation between special linear and general algebraic cobordism MSL and MGL},
year = {2023},
eprint = {2310.15721},
archivePrefix = {arXiv},
primaryClass = {math.AT}
}

@article{MV99,
  title = {A1-Homotopy Theory of Schemes},
  author = {Morel, Fabien and Voevodsky, Vladimir},
  date = {1999-12},
  journaltitle = {Publications mathématiques de l'IHÉS},
  shortjournal = {Publications Mathématiques de L’Institut des Hautes Scientifiques},
  volume = {90},
  number = {1},
  pages = {45--143},
  issn = {0073-8301, 1618-1913},
  doi = {10.1007/BF02698831},
  url = {http://link.springer.com/10.1007/BF02698831},
  urldate = {2020-11-09},
  langid = {english}
}

@article {Hoyois-KQ,
    AUTHOR = {Hoyois, Marc and Jelisiejew, Joachim and Nardin, Denis and
              Yakerson, Maria},
     TITLE = {Hermitian {K}-theory via oriented {G}orenstein algebras},
   JOURNAL = {J. Reine Angew. Math.},
  FJOURNAL = {Journal f\"ur die Reine und Angewandte Mathematik. [Crelle's
              Journal]},
    VOLUME = {793},
      YEAR = {2022},
     PAGES = {105--142},
      ISSN = {0075-4102,1435-5345},
   MRCLASS = {14F42 (14C35 19G38)},
  MRNUMBER = {4513165},
MRREVIEWER = {Guillermo\ Corti\~nas},
       DOI = {10.1515/crelle-2022-0063},
       URL = {https://doi.org/10.1515/crelle-2022-0063},
}

@article {norms,
    AUTHOR = {Bachmann, Tom and Hoyois, Marc},
     TITLE = {Norms in motivic homotopy theory},
   JOURNAL = {Ast\'erisque},
  FJOURNAL = {Ast\'erisque},
    NUMBER = {425},
      YEAR = {2021},
     PAGES = {ix+207},
      ISSN = {0303-1179,2492-5926},
      ISBN = {978-2-85629-939-5},
   MRCLASS = {14F42 (19E15)},
  MRNUMBER = {4288071},
MRREVIEWER = {Jon\ Eivind\ Vatne},
       DOI = {10.24033/ast},
       URL = {https://doi.org/10.24033/ast},
}

@misc{cells,
author = {Wendt, Matthias},
title = {More examples of motivic cell structures},
year = {2010},
eprint={1012.0454},
archivePrefix={arXiv},
primaryClass={math.AG}
}
\end{document}